\documentclass[11pt]{amsart}

\title[Essential surfaces in handlebody-knot exteriors]
{Essential surfaces of non-negative Euler characteristic in genus two handlebody exteriors}

\author{Yuya Koda}
\author{Makoto Ozawa  \\ \\
with an appendix by Cameron Gordon}

\thanks{The first-named author is supported in part by
Grant-in-Aid for Young Scientists (B) (No. 20525167), Japan Society for the Promotion of Science, and by 
JSPS Postdoctoral Fellowships for Research Abroad.}

\address{
Mathematical Institute \newline
\indent Tohoku University, Sendai, 980-8578, Japan \newline 
\indent and \newline
\indent (Temporary) Dipartimento di Matematica  \newline
\indent Universit\`{a} di Pisa, Largo Bruno Pontecorvo 5, 56127 Pisa, Italy}
\email{koda@math.tohoku.ac.jp}

\thanks{The second-named author is supported in part by
Grant-in-Aid for Scientific Research (C) (No. 23540105), Japan Society for the Promotion of Science.}

\address{
Department of Natural Sciences, Faculty of Arts and Sciences \newline
\indent Komazawa University, 1-23-1 Komazawa, Setagaya-ku, Tokyo, 154-8525, Japan}
\email{w3c@komazawa-u.ac.jp}

\address{
Department of Mathematics \newline
\indent The University of Texas at Austin, Austin, TX 78712, USA}
\email{gordon@math.utexas.edu}

\usepackage{amsfonts,amsmath,amssymb,amscd}

\usepackage{amsthm}
\usepackage{latexsym}
\usepackage[dvips]{graphicx}
\usepackage[dvips]{psfrag}
\usepackage[dvips]{color}
\def\qed{\hfill $\Box$} 

\theoremstyle{plain}
\newtheorem*{theorem*}{Theorem}
\newtheorem*{lemma*} {Lemma}
\newtheorem*{corollary*} {Corollary}
\newtheorem*{proposition*}{Proposition}
\newtheorem*{conjecture*}{Conjecture}
\newtheorem{theorem}{Theorem}[section]
\newtheorem{lemma}[theorem]{Lemma}
\newtheorem{corollary}[theorem]{Corollary}

\newtheorem{conjecture}[theorem]{Conjecture}
\newtheorem{claim}{Claim}

\theoremstyle{remark}

\newtheorem*{remark}{Remark}

\newtheorem*{notation}{Notation}
\newtheorem*{example}{Example}

\theoremstyle{definition}

\newcommand{\Integer}{\mathbb{Z}}

\newcommand{\Int}{\mathrm{Int}}

\newcommand{\Nbd}{N}



\begin{document}
\maketitle

\begin{abstract}
We provide a classification of the essential surfaces of 
non-negative Euler characteristic in the exteriors of 
genus two handlebodies embedded in the 3-sphere. 
\end{abstract}

\vspace{1em}

\begin{small}
\hspace{2em}  \textbf{2010 Mathematics Subject Classification}:
57M25; 57M15


\hspace{2em} 
\end{small}

\section*{Introduction}

%

As is well-known, 
the set of knots in the 3-sphere is classified 
into four classes;  
the {\it trivial knot}, {\it torus knots}, 
{\it satellite knots} and {\it hyperbolic knots}, 
depending on the types of the essential surfaces of 
non-negative Euler characteristic 
lying in their exteriors. 
The trivial knot is the only knot that 
contains an essential disk 
in its exterior, while the torus 
knot exteriors contain 
essential annuli but 
do not contain essential tori. 
The class of satellite knots consists of 
knots admitting essential tori. 
Classical studies on knots prove that the essential annuli in 
the exterior of torus knots or satellite knots are very limited, 
that is, 
each of them is either a cabling annulus or 
that which can be extended to decomposing spheres (cf. Lemma \ref{lem:BZ85}). 
The class of hyperbolic knots 
consists of knots whose exteriors are {\it simple}, 
that is, do not admit 
any essential surfaces of Euler characteristic 
at least zero. 
By Thurston's Hyperbolization Theorem \cite{Thu82, Mor84, Ota96, Ota98, Kap01}, 
the complement of each hyperbolic knot admits 
a complete hyperbolic metric of finite volume. 
A great many studies on knots have been based on this classification. 

A genus $g$ handlebody $V$ 
embedded in the 3-sphere $S^3$, where 
$g$ is a non-negative integer, is called a 
{\it genus $g$ handlebody-knot} and denoted by $(S^3, V)$. 
When $g$ equals one, the study of handlebody-knots 
coincides with the classical knot theory. 
On the other hand, 
the study of handlebody-knots 
whose exteriors are also handlebodies is related to 
the theory of Heegaard splittings. 
By Thurston's Hyperbolization Theorem again, 
the exterior $E(V)$ 
of handlebody-knot $V$ of genus at least two 
is simple if and only if $E(V)$ admits a hyperbolic structure 
with totally geodesic boundary. 
Otherwise, 
the configurations of 
essential surfaces of non-negative Euler characteristic in 
the exterior $E(V)$ are much more complicated in general 
compare to the case of knots. 
The aim of this paper is to classify these essential surfaces 
in the exteriors of genus two handlebody-knots. 
In fact, we classify, without overlap, the essential disks into three types 
(cf. Section \ref{sec:Classification of essential disks in the exteriors of genus two handlebodies embedded in the 3-sphere}), 
the essential annuli into four types (cf. Section \ref{sec:Classification of essential annuli in the exteriors of genus two handlebodies embedded in the 3-sphere}), 
the essential M\"obius bands into two types 
(cf. Section 
\ref{sec:Classification of essential Mobius bands in 
the exteriors of genus two handlebodies embedded in the 3-sphere}), and  
the essential tori into three types 
(cf. Section 
\ref{sec:Classification of essential tori in 
the exteriors of genus two handlebodies embedded in the 3-sphere}). 
This should be contrasted with the case of knots; 
the essential annuli, for example, in knot exteriors 
can be classified into two types, as was mentioned above. 
To obtain the above classification, we fully use the 
results on essential planar surfaces and punctured tori 
properly embedded in the exteriors of knots, which are strongly 
related to the study of Dehn surgeries on knots in the 3-sphere 
that produce reducible or toroidal 3-manifolds.  

In \cite{Mot90}, Motto gave an infinite family of 
genus two handlebody-knots, and using essential annuli lying in their exteriors, 
he showed that the handlebody-knots in the family 
are mutually distinct 
whereas they have homeomorphic exteriors. 
In \cite{LL12}, Lee and Lee provided other infinite families of 
genus two handlebody-knots such that 
the handlebody-knots in each of the families 
are mutually distinct whereas they have homeomorphic exteriors. 
Detailed description of 
essential annuli in the exteriors of the handlebody-knots 
again played an important role in their paper. 
Also, in \cite{EO12}, Eudave-Mu\~{n}oz and 
the second-named author determined 
essential annuli that can be extended to 2-decomposing spheres 
in tunnel number one, genus two handlebody-knot exteriors and 
they characterized their summands by $2$-decomposing 
spheres. 
Each of the above families of essential annuli 
is entirely contained in one type of 
the essential surfaces studied in this paper. 

On the other hand, the first-named author 
defined in \cite{Kod11} the {\it symmetry group} 
of a handlebody-knot. 
This is the group of isotopy classes of 
self-homeomorphisms of $S^3$ leaving the handlebody-knot invariant. 
When the exterior of a genus two handlebody-knot 
is boundary-reducible or simple, 
a finite presentation of its symmetry group 
can be obtained following \cite{Goe33, Sch04, Akb08, Cho08, Kod11}.  
However, apart from a few examples, 
the symmetry groups of the remaining handlebody-knots 
still remain unknown. 
The result in this paper 
would be a beginning step to developing the study of the symmetry groups. 
 
In \cite{IKO12}, Ishii, Kishimoto and the second-named author  
showed the unique decomposition theorem 
with respect to a special kind of 
$2$-decomposing spheres for handlebody-knots of arbitrary genus 
whose exteriors are boundary-irreducible. 
In an appendix of the paper, we prove the same uniqueness theorem 
for arbitrary handlebody-knots. 

Throughout this paper, 
we will work in the piecewise linear category.

\begin{notation}
Let $X$ be a subset of a given polyhedral space $Y$. 
Throughout the paper, we will denote the interior of 
$X$ by $\Int \thinspace X$
and the number of components of $X$ by $\# X$. 
We will use $\Nbd (X; Y)$ to denote a closed regular neighborhood of $X$ in $Y$. 
If the ambient space $Y$ is clear from the context, 
we denote it briefly by $\Nbd (X)$. 
Let $M$ be a 3-manifold. 
Let $L \subset M$ be a submanifold with or without boundary. 
When $L$ is 1 or 2-dimensional, we write 
$E(L) = M \setminus \Int \thinspace \Nbd (L)$. 
When $L$ is of 3-dimension, we write 
$E(L) = M \setminus \Int \thinspace L$. 
We shall often say 
surfaces, compression bodies, 
e.t.c. in an ambient manifold 
to mean the isotopy classes of them. 
\end{notation}


\noindent {\bf Acknowledgments.}
The authors would like to thank Mario Eudave-Mu\~noz 
for his valuable comments related 
to Lemma \ref{lem:handlebodies and Eudave-Munoz knots}. 
This work was carried out while the first-named author was visiting 
Universit\`a di Pisa as a 
JSPS Postdoctoral Fellow for Reserch Abroad. 
A part of the work was done while the second-named author was visiting 
Universit\`a di Pisa. 
They are grateful to the university and its staffs for 
the warm hospitality. 

\section{Preliminaries}
\label{sec:Preliminaries}

Let $M$ be a compact orientable 3-manifold. 
Let $F$ be an orientable 
(possibly not connected) surface 
properly embedded in $M$. 
A disk $D$ embedded in $M$ is called 
a {\it compressing disk} for $F$ if $D \cap F = \partial D$ and 
$\partial D$ is an essential simple closed curve on $F$. 
A disk $D$ embedded in $M$ is called a 
{\it boundary-compressing disk} for $F$ if $D \cap F \subset \partial D$ 
is a single essential arc 
on $F$ and 
$D \cap \partial M = \partial D \setminus \Int \thinspace (D \cap F)$. 
The surface $F$ is said to be 
{\it incompressible} ({\it boundary-incompressible}, respectively) 
if there exists no compressing disk (boundary-compressing
disk, respectively) for $F$. 
The surface $F$ is said to be {\it essential} 
if $F$ is incompressible, boundary-incompressible and not boundary parallel. 
A connected non-orientable surface $F'$ 
properly embedded in $M$ is said to be {\it essential} 
if the {\it frontier} of $\Nbd(F'; M)$, that is, the closure of 
$\partial \Nbd (F'; M) \setminus \partial M$, is essential. 

We recall that a {\it handlebody} is 
a compact orientable 3-manifold 
containing pairwise disjoint essential disks 
such that the manifold obtained by cutting along the disks is a 3-ball. 
The {\it genus} of a handlebody is defined to be the genus of its boundary 
surface. 
The following well-known fact will 
be needed later. See e.g. \cite{Jac80}. 
\begin{lemma}
\label{lem:essential surfaces in a handlebody}
Let $F$ be an essential surface in a handlebody. 
Then $F$ is a disk. 
\end{lemma}

The essential annuli in knot exteriors are classified as follows. 
See e.g. \cite{BZ85}. 
\begin{lemma}
\label{lem:BZ85}
Let $K$ be a knot in $S^3$. If $E(K)$ contains an 
\label{lem:BZ85}
essential annulus $A$, then exactly one of the following holds: 
\begin{enumerate} 
\item
$K$ is a torus knot or a cable knot and $A$ is its cabling annulus; 
\item
$K$ is a composite knot and $A$ can be extended to a decomposing sphere for
$K$. 
\end{enumerate}
\end{lemma}
We note that the above lemma can be generalized as 
a classification of the essential annuli in the exteriors of links in $S^3$. 
In fact, if $A$ is an essential annulus in the exterior of 
a link, 
then 
$A$ is a cabling annulus, 
$A$ can be extended to a decomposing sphere, or 
$A$ connects two components of the link, where at least one of the boundary components of $A$ has a
meridional or integral boundary-slope. 
%
%
%

As a direct corollary of Lemma \ref{lem:BZ85}, we can also 
classify the essential M\"obius bands 
in knot exteriors as follows: 

\begin{lemma}
\label{lem:essential Mobius bands in knot exteriors}
Let $K$ be a knot in $S^3$. If $E(K)$ contains 
an essential M\"{o}bius band $F$, then $K$ is either 
an $(n,2)$-torus knot or 
an $(n,2)$-cable knot for an odd integer $n$, and 
the frontier of $N(F)$ 
satisfies $(1)$ in Lemma $\ref{lem:BZ85}$. 
\end{lemma}

In Sections 
\ref{sec:Classification of essential annuli in the exteriors of genus two handlebodies embedded in the 3-sphere} and 
\ref{sec:Classification of essential Mobius bands in 
the exteriors of genus two handlebodies embedded in the 3-sphere}, 
we obtain the same type of classifications as 
Lemmas \ref{lem:BZ85} and \ref{lem:essential Mobius bands in knot exteriors}, 
respectively, for genus two handlebody-knots.

Let $M$ be a compact orientable 3-manifold. 
Let $F$ be an orientable  
surface (possibly not connected) properly embedded in $M$.  
Let $D$ be a compressing disk for $F$. 
Then we have a new proper surface $F'$ by cutting $F$ 
along $\partial D$ and pasting two copies of $D$ to it. 
We say that {\it $F'$ is 
obtained by 
compressing $F$ along $D$}. 

Let $M$ be a 3-manifold. 
We recall that $M$ is said to be {\it reducible} if it contains  
a sphere that does not bound a 3-ball in $M$. 
Otherwise, $M$ is said to be {\it irreducible}. 
Also, $M$ is said to be {\it boundary-reducible} 
if it contains an essential disk.  
Otherwise, $M$ is said to be {\it boundary-irreducible}. 

\begin{lemma}
\label{lem:annulus with essential boundary is esential}
Let $M$ be a compact, orientable, irreducible, boundary-irreducible $3$-manifold 
such that $\partial M$ is a closed surface of genus at least two. 
Let $A$ be an annulus properly enbedded in $M$. 
If each component of $\partial A$ is essential on $\partial M$ and 
$A$ is not parallel to the boundary of $M$, then 
$A$ is essential in $M$.  
\end{lemma}
\begin{proof}
Assume that each component of $\partial A$ is 
non-trivial on $\partial M$ and 
that $A$ is not parallel to $\partial M$. 
If $A$ admits a compressing disk $D_1$ in $M$, 
then each of the disks obtained by 
compressing $A$ along $D_1$ is an essential disk in $M$. 
This contradicts the assumption that 
$M$ is boundary-irreducible. 
Thus it suffices to show that $A$ is boundary-incompressible. 
Assume for contradiction that $A$ admits a 
boundary-compressing disk $D_2$ in $M$. 
Let $D$ be the disks obtained by 
boundary-compressing $A$ along $D_2$. 
We will show that $D$ is an essential disk in $M$. 
Set $\gamma = \partial D_2 \cap \partial M$. 
We note that $\partial D$ is the component of 
$\partial N(\partial A \cup \gamma; \partial M)$ that is not parallel to 
neither component of $\partial A$. 
If the two simple closed curves $\partial A$ are not parallel 
on $\partial M$, then $\partial D$ is not trivial on $\partial M$. 
Hence $D$ is an essential disk in $M$. 
Assume that $\partial A$ consists of 
parallel simple closed curves on $\partial M$. 
Let $A'$ be the sub-annulus of $\partial M$ such that 
$\partial A' = \partial A$. 
If $\gamma$ is not contained in $A'$, then 
$\partial D$ is not trivial on $\partial M$. 
Hence $D$ is an essential disk in $M$. 
If $\gamma$ is contained in $A'$, then 
$D_2$ is an essential disk in a component $N$ of $S^3$ cut off by the torus $A \cup A'$. 
This implies that $N$ is a solid torus and $D_2$ is its meridian disk. 
Moreover, $\partial D_2$ intersects each component of $\partial A$ once and transversely. 
Hence $A$ is parallel to $\partial M$ through $N$. 
This is a contradiction. 
\end{proof}

Let $M$ be a compact orientable 3-manifold. 
Let $F$ be an orientable  
surface (possibly not connected) properly embedded in $M$.  
An annulus $A$ embedded in $M$ is called 
a {\it peripherally compressing annulus} for $F$ 
if $A \cap F$ is a single essential simple 
closed curve on $F$ 
and $A \cap \partial M = \partial A \setminus (A \cap F)$ 
is a single essential simple closed curve on $\partial M$. 
We note that 
a peripherally compressing annulus is called an 
{\it accidental annulus} when it is considered in a knot exterior. 
See e.g. \cite{IO00}. 
Let $A$ be a peripherally compressing annulus for $F$. 
Then we have a new proper surface $F'$ by cutting $F$ 
along $F \cap A$ and pasting two copies of $A$ to it. 
We say that {\it $F'$ is obtained by 
peripherally compressing $F$ along $A$}. 

\begin{lemma}
\label{lem:perhipherally compressing a torus}
Let $M$ be a compact, orientable, 
irreducible $3$-manifold such that 
$\partial M$ is a torus. 
Let $T$ be an essential torus in $M$. 
Let $A$ be a peripherally compressing annulus for $T$. 
Then the annulus obtained by peripherally compressing 
$T$ along $A$ is essential in $M$. 
\end{lemma}
\begin{proof}
Let $T'$ be the annulus obtained by peripherally compressing 
$T$ along $A$. 
Assume that there exists a compressing disk $D_1$ for $T'$. 
We can isotope $D_1$ so that 
$\partial D_1 \cap \Nbd(A) = \emptyset$. 
Then $\partial D_1$ is parallel to $A \cap T$, otherwise 
$\partial A$ is not essential on the annulus $T'$. 
Since $A \cap T$ is essential on $T$, 
$D_1$ is a compressing disk for $T$. 
This is a contradiction. 

Assume that there exists a boundary-compressing disk $D_2$ 
for $T'$. 
We note that the two components $\partial T'$ are parallel on 
the boundary of $M$. 
Let $A'$ be the sub-annulus of $\partial M$ such that 
$\partial A' = \partial A$ and $\partial D_2 \cap \partial M \subset A'$. 
Since $M$ is irreducible, the component $N$ of $M$ 
cut off by $T'$ which contains $D_2$ is a solid torus and 
$D_2$ is its meridian disk. 
Then $T'$ and $A'$ are parallel through $N$ since each 
component of $\partial A$ intersects $\partial D_2$ once and transversely. 
\begin{figure}[!hbt]
\centering
\includegraphics[width=12cm,clip]{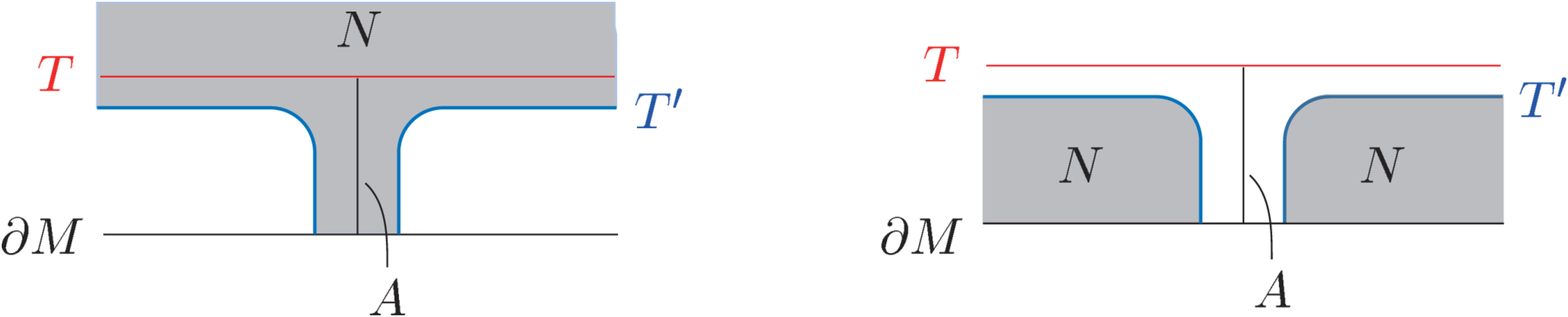}
\caption{}
\label{annulus_compression}
\end{figure}
Now, if $N$ contains $A$, $T$ is compressible in $N \subset M$ 
(See the left-hand side of Figure \ref{annulus_compression}). 
Otherwise, $T$ is parallel to $\partial M$ 
(See the right-hand side of Figure \ref{annulus_compression}).  
Therefore both cases contradicts the assumption that 
$T$ is essential in $M$. This completes the proof. 
\end{proof}

Let $(S^3, V)$ be a handlebody-knot. 
We say that $(S^3, V)$ is {\it trivial} 
if $E(V)$ is also a handlebody. 
A 2-sphere $S$ in $S^3$ is called 
an {\it $n$-decomposing sphere} for $(S^3, V)$ if 
$S \cap V$ consists of $n$ essential disks in $V$, and
$S \cap E(V)$ is an essential surface in $E(V)$. 
A handlebody-knot $(S^3, V)$ is said to be {\it $n$-decomposable} 
if it admits an $n$-decomposing sphere. 
A 1-decomposable handlebody-knot which is sometimes said to be 
{\it reducible}. 
Otherwise, it is said to be {\it irreducible}. 
We note that, by 
Lemma \ref{lem:essential surfaces in a handlebody}, 
tirivial handlebody-knots are not $n$-decomposable for 
$n>1$. 
It is proved in \cite{Tsu75, BF12} that a handlebody-knot $(S^3, V)$ 
of genus two is 1-decomposable if and only if its exterior 
$E(V)$ is $\partial$-reducible, i.e. 
$\partial E(V)$ is compressible in $E(V)$. 
See e.g. \cite{Ish08, IKO12} and the references given there 
for more details.

\section{Classification of the essential disks in genus two handlebody-knot exteriors}
\label{sec:Classification of essential disks in the exteriors of genus two handlebodies embedded in the 3-sphere}

We first review the notion of {\it characteristic compression body} 
introduced in \cite{Bon83}. 
Let $M$ be an irreducible compact 3-manifold with boundary and let 
$\mathcal{D}$ be the union 
of mutually disjoint compression disks for $\partial M$. 
Let $W$ be the union of 
$\Nbd (\mathcal{D} \cup \partial M; M)$ and all the components 
of $M \setminus \Int \thinspace (\Nbd (\mathcal{D} \cup \partial M; M))$ 
that are 3-balls.
Then we call $W$ a {\it compression body} for $\partial M$. 
Also, $\partial_+ W = \partial M \subset \partial W$ is called 
the {\it exterior boundary} of
$W$ and 
$\partial_- W = \partial W \setminus \partial_+ W$ is called 
the {\it interior boundary} of 
$W$. 
A {\it characteristic compression body} $W$ of $M$ is 
a compression body for $\partial M$ such that 
$M \setminus \Int \thinspace W$ is boundary-irreducible. 
Here, we remark that, 
if $W$ is a characteristic compression body, 
every compressing disk $D$ for $\partial M$ can be 
isotoped so that $D \subset W$. 
We also remark that any closed incompressible surface in 
$W$ is parallel to a sub-surface of $\partial_- W$ (see e.g. \cite{Bon83}). 
\begin{theorem}[\cite{Bon83}]
An irreducible compact $3$-manifold with boundary
has a unique characteristic compression body. 
\end{theorem}

Let $(S^3,V)$ be a genus two handlebody-knot. 
Let $W$ be the characteristic compression body for of $E(V)$. 
We classify $V$ into the following four types: 
\begin{description}
\item[($i$)] 
$\partial_- W$ is a closed orientable surface of genus two; 
\item[($ii$)] 
$\partial_- W$ consists of two tori;
\item[($iii$)] 
$\partial_- W$ is a torus; 
\item[($iv$)] 
$\partial_- W = \emptyset$.
\end{description}

\begin{figure}[!hbt]
\centering
\includegraphics[width=14.5cm,clip]{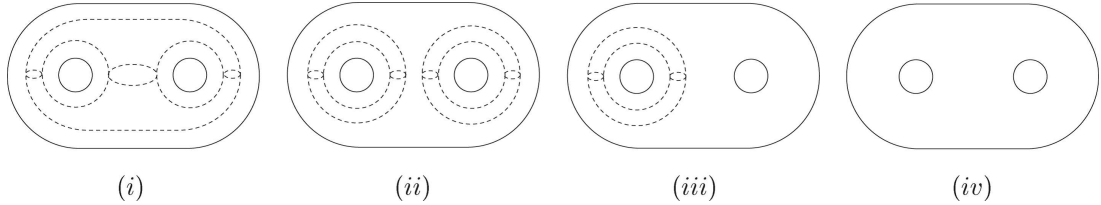}
\caption{The four types of characteristic compression bodies 
$W \subset E(V)$.}
\label{characteristic_compression_bodies}
\end{figure}
Let $(S^3, V)$ be a genus two handlebody-knot. 
As we mentioned in Section \ref{sec:Preliminaries}
$V$ is of type $(i)$ if and only if $V$ is not $1$-decomposable. 
We also note that $V$ is of type $(iv)$ if and only if $V$ is trivial. 




Let $X$ be a handlebody of genus at least $1$. 
A simple closed curve $l$ on $\partial X$ 
is said to be {\it primitive} with respect to $X$ 
if there exists an essential disk $E$ in $X$ such that 
$\partial E$ and $l$ have 
a single transverse intersection on $\partial X$. 

Let $(S^3, V)$ be a genus two handlebody-knot. 
We introduce the following three types of essential disks in 
$E(V)$. 
\begin{description}
\item[{\rm Type $1$ ($1$-decomposing sphere type)}]
An essential disk $D$ in $E(V)$ is called 
a {\it Type $1$ disk} if 
$\partial D$ bounds an essential disk $D'$ in $V$. 
Here we remark that 
$D \cup D'$ becomes a $1$-decomposing sphere for $V$; 
\item[{\rm Type $2$ (primitive disk type)}]
An essential disk $D$ in $E(V)$ is called 
a {\it Type $2$ disk} if 
$\partial D$ is primitive with respect to $V$; 
\item[{\rm Type $3$ (unknotting tunnel type)}]
An essential disk $D$ in $E(V)$ is called 
a {\it Type $3$ disk} if 
there exists a tunnel number one $2$-component link $l_1 \sqcup l_2$ 
and an unknotting tunnel $\tau$ of it such that 
\begin{itemize}
\item
$l_1$ is a trivial knot; 
\item
there exists a re-embedding 
$h: E(l_1) \to S^3$ 
such that $V = h(E(l_1 \cup l_2 \cup \tau))$ 
and $D = h (D_*)$, where $D_*$ is the co-core of 
the $1$-handle $\Nbd(\tau ; E(l_1 \sqcup l_2))$ attached to 
$\Nbd(l_1 \sqcup l_2)$.  
\end{itemize}
\end{description}

\begin{example}
Let $L = l_1 \sqcup l_2$ be the Whitehead link 
and $\tau$ be its unknotting tunnel as illustrated in 
the left-hand side of Figure \ref{example_disk_1}. 
Let $h: E(l_1) \to S^3$ be the re-embedding 
such that $h(E(l_1))$ is 
a thicken trefoil. 
Then $V = h(E(l_1 \cup l_2 \cup \tau))$ is a 
genus two handlebody-knot 
and the image $D$ of the co-core 
the $1$-handle $\Nbd(\tau ; E(l_1 \cup l_2))$ 
becomes an essential disk in $E(V)$ 
as shown in 
the right-hand side of Figure \ref{example_disk_1}.
\begin{figure}[!hbt]
\centering
\includegraphics[width=10cm,clip]{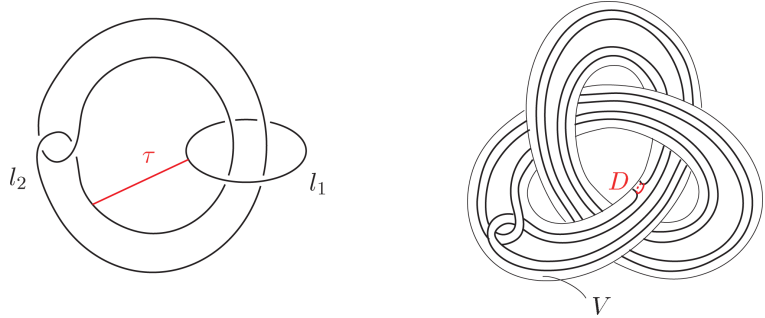}
\caption{A Type 3 essential disk.}
\label{example_disk_1}
\end{figure}
\end{example}

We remark that if $D$ is an essential disk 
in the exterior $W$ of the trivial 
genus two handlebody-knot $V$, then 
$D$ is the dual disk of an unknotting tunnel of 
the tunnel number one knot or link which is the 
core of $W \setminus \Int \thinspace \Nbd(D; W)$. 

\begin{theorem}
\label{thm:classification of essential disks}
Let $(S^3,V)$ be a non-trivial 
genus two handlebody-knot. 
Then each essential disk $D$ in the exterior of $V$ 
belongs to exactly one of the above three Types. 
\end{theorem}
\begin{proof}
Let $D$ be an essential disk in $E(V)$. 
By definition, we may easily check that 
$D$ can not belong to 
more than one type. 

Let $W$ be the characteristic compression body of $E(V)$. 
We first consider the case where $V$ is of type $(ii)$. 
Set $\partial_- W = T_1 \sqcup T_2$, where each of 
$T_1$ and $T_2$ is a torus. 
It is clear that $D$ is separating in $E(V)$. 
Since $T_1 \sqcup T_2$ is incompressible in $E(V \cup W)$ and 
compressible in $S^3$, $T_1 \sqcup T_2$ is compressible in $V \cup W$. 
Let $D' \subset V \cup W$ be a compressing disk for $T_1$ 
and $S'$ be a sphere obtained by compressing $T_1$ along $D'$. 
We note that $S'$ is an essential sphere in $V \cup W$, otherwise 
$V \cup W$ is a solid torus, which is a contradiction. 
By Haken's lemma \cite{Hak68}, 
there exists an essential sphere $S''$ in $V \cup W$ such that $S'' \cap V$ 
is a single disk. 
This implies that $S''$ is a 1-decomposing sphere for $V$. 
By Lemma 3.1 of \cite{Kod11}, $S'' \cap W$ 
is a unique compressing disk 
of $\partial E(V)$ in $E(V)$, which implies $S'' \cap W$ 
is isotopic to $D$ in $E(V)$. 
Therefore $\partial D$ bounds a disk 
(parallel to $S'' \cap V$) in $V$, 
hence $D$ is a Type 1 disk. 

In the following we shall consider the case 
where $V$ is of type $(iii)$. 
In this case $V \cup W$ is a solid torus since 
$\partial_- W$ bounds a solid torus in $S^3$ while 
$E(V \cup W)$ is not a solid torus.  

Suppose that $D$ is non-separating in $E(V)$. 
Then there exists a simple arc $\gamma$ properly embedded in $W$ such that 
\begin{itemize}
\item
$\gamma$ intersects $D$ once and transversely; and 
\item
$\gamma \cup \partial_-W$ is a spine of $W$. 
\end{itemize}
See the left-hand side of Figure \ref{fig:characteristic_compression_body_type_iii}. 
\begin{figure}[!hbt]
\centering
\includegraphics[width=10cm,clip]{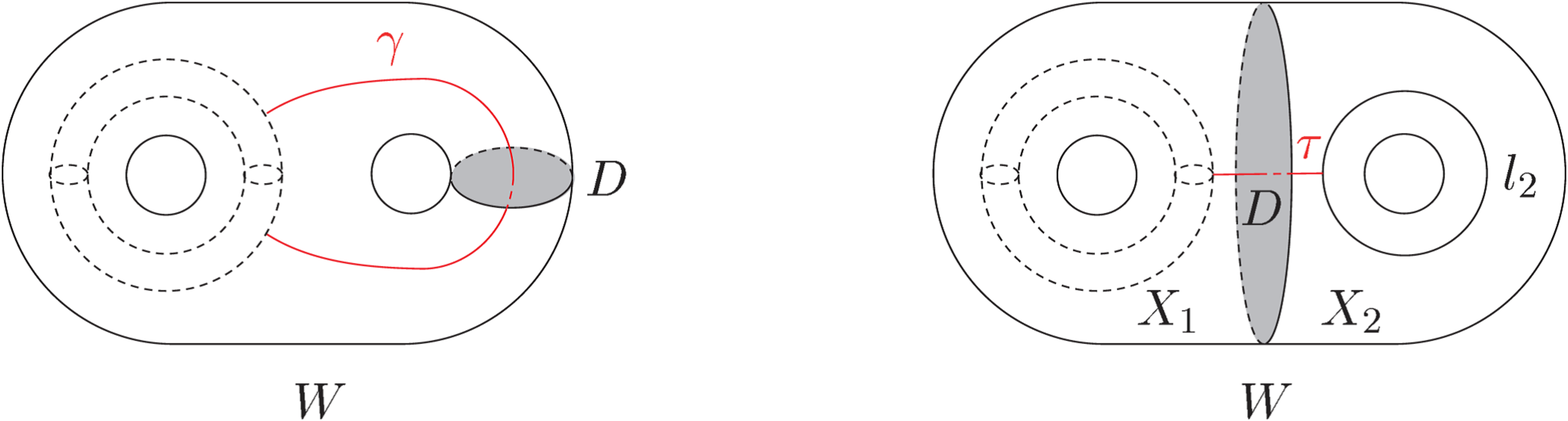}
\caption{}
\label{fig:characteristic_compression_body_type_iii}
\end{figure}
Since $V \cap W$ is a solid torus and 
$V \cup W \setminus \Int \thinspace \Nbd(\gamma) \cong V$ is a genus two handlebody, 
it follows from \cite{Gor87} that $\gamma$ is unknotted in $V \cup W$, 
that is, there exists a disk $E$ in $V \cup W$ such that 
$E \cap \gamma = \partial E \cap \gamma = \gamma$ and 
$E \cap \partial_- W = \partial E \setminus \Int \thinspace \gamma$. 
This implies that $D$ is a Type 2 disk. 

Suppose that $D$ is separating in $E(V)$. 
We set 
$W \setminus \Int \thinspace \Nbd(D; W) = X_1 \sqcup X_2$, 
where $X_1 \cong T^2 \times [0,1]$ and $X_2$ is a solid torus. 
Let $l_2$ be the core of $X_2$. 
Then there exists a simple arc $\tau$ in $W$ such that 
\begin{itemize}
\item
$\tau$ connects $\partial_- W$ and $l_2$; 
\item
$(\Int \thinspace \tau ) \cap (\partial_- W \cup l_2) = \emptyset$; 
\item
$\tau$ intersects $D$ once and transversely; and 
\item
$\Gamma = \tau \cup l_2 \cup \partial_-W$ is a spine of $W$. 
\end{itemize}
See the right-hand side of Figure \ref{fig:characteristic_compression_body_type_iii}. 
We re-embed the solid torus $V \cup W$ into $S^3$ by a map 
$\iota : V \cup W \to S^3$ so that 
$E(\iota(V \cup W))$ is a solid torus. 
Let $l_1$ be the core of $E(\iota(V \cup W))$. 
Then $l_1 \cup \iota(l_2)$ is a tunnel number one link with 
an unknotting tunnel $\iota(\tau)$, hence $D$ is a 
Type 3 disk. 
This completes the proof. 
\end{proof}

\section{Classification of the essential annuli in genus two handlebody-knot exteriors}
\label{sec:Classification of essential annuli in the exteriors of genus two handlebodies embedded in the 3-sphere}

In this section, we provide a classification of 
the essential annuli in the exteriors of genus 
two handlebody-knots.  
Essential annuli in one of the four Types in the classification 
are described using {\it Eudave-Mu\~noz knots}. 
We quickly review the definition and important properties of 
this class of knots.

In \cite{Eud97} Eudave-Mu\~noz provided 
an infinite family of hyperbolic knots $k(l, m, n, p)$, where 
either $n$ or $p$ is equal to $0$,  
that admit non-integral toroidal surgeries. 
The knots are now called {\it Eudave-Mu\~noz knots}. 
The construction of the knot $k(l, m, n, p)$ can be 
briefly explained as follows. 
Let $(B, T)$ be the two-string tangle 
shown in Figure \ref{EM_knots}. 
In the figure, $(B, T)$ lies outside of the 
small circle depicted in the middle. 
Then the double branched cover of the tangle $(B, T)$ is the 
exterior of the Eudave-Mu\~noz knot $k(l, m, n, p)$. 
\begin{figure}[!hbt]
\centering
\includegraphics[width=14.5cm,clip]{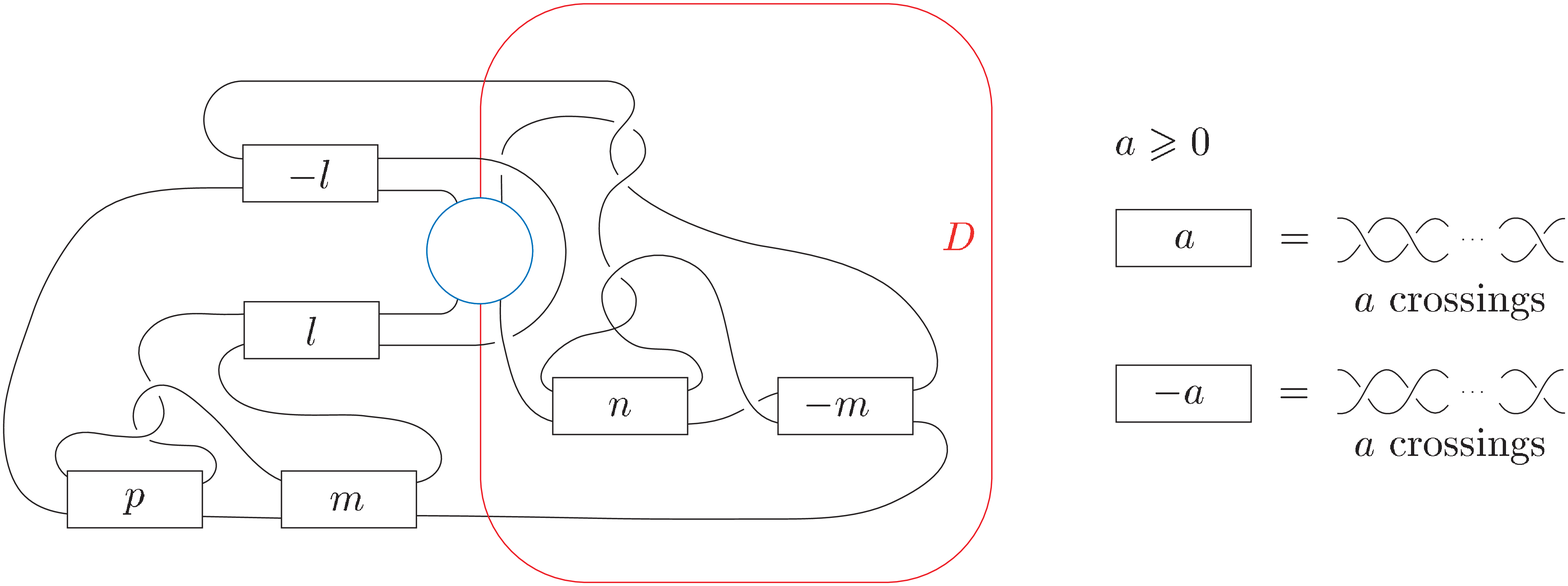}
\caption{}
\label{EM_knots}
\end{figure}
We note that the $(-2, 3, 7)$-pretzel knot, which is one of the 
most famous example of knots that admits non-integral 
toroidal Dehn surgeries, is $k(3, 1, 1, 0)$. 
In \cite{Eud02}, a non-integral toroidal slope $r$ 
for $k(l, m, n, p)$ is described in terms of the parameters as 
\[r = l(2m - 1)(1 - lm) + n(2lm - 1)^2 - 1/2\]
 for $k(l, m, n, 0)$ and 
\[r = l(2m - 1)(1 - lm) + p(2lm - l - 1)^2 - 1/2\]
 for $k(l, m, 0, p)$. 
The slope $r$ is obtained as a lift of the circle $\partial D$, 
where the disk $D$ is depicted as a red arc in Figure \ref{EM_knots}. 
Gordon and Luecke \cite{GL04} proved that these are the only 
hyperbolic knots 
which admit non-integral toroidal surgeries.

\begin{theorem}[\cite{GL04}] 
\label{thm:classification of hyperbolic knots admitting a non-integral toroidal surgery}
Let $K$ be a hyperbolic knot in $S^3$ that admits a
non-integral toroidal surgery. 
Then $K$ is one of the Eudave-Mu\~noz knots 
and the toroidal slope is $r$ described above.
\end{theorem}

\begin{lemma}
\label{lem:handlebodies and Eudave-Munoz knots}
Let $K$ be an Eudave-Mu\~noz knot 
and let $P$ be an incompressible twice-punctured torus 
properly embedded 
in $E(K)$ such that $\partial P$ consists of 
the two parallel toroidal slopes of $K$. 
Then $P$ cuts off $E(K)$ into two handlebodies of genus two. 
\end{lemma}
\begin{proof}
Let $K=k(l, m, n, p)$. 
Let $(B, T)$ and $D$ be the tangle and the disk, respectively, 
as shown in Figure \ref{EM_knots}. 
Let $p: E(K) \to B$ be the double branched covering of 
$(B, T)$. 
Then we have $P = p^{-1}(D)$. 
Since the disk $D$ cuts off $(B,T)$ into two trivial 3-string tangles 
$(B_1, T_1)$ and $(B_2, T_2)$, 
$P$ cuts off $E(K)$ into two genus two handlebodies 
$p^{-1}(B_1)$ and $p^{-1}(B_2)$. 
\end{proof}

Let $(S^3, V)$ be a genus two handlebody-knot. 
We provide a list of annuli properly embedded in $E(V)$. 

\begin{description}

\item[{\rm Type 1 (2-decomposing sphere type)}]
Let $\Gamma \subset S^3$ be a spatial handcuff-graph.  
Let $S$ be a sphere in $S^3$ that intersects $\Gamma$ 
in exactly one edge of $\Gamma$ twice and transversely. 
Set $V = \Nbd(\Gamma)$. 
We call $A = S \setminus \Int \thinspace V$ 
a {\it Type $1$ annulus} for 
the handlebody-knot $(S^3, V)$ if 
$A$ is not parallel to the boundary of $V$. 
See Figure \ref{type1_annulus}. 
\begin{figure}[!hbt]
\centering
\includegraphics[width=8cm,clip]{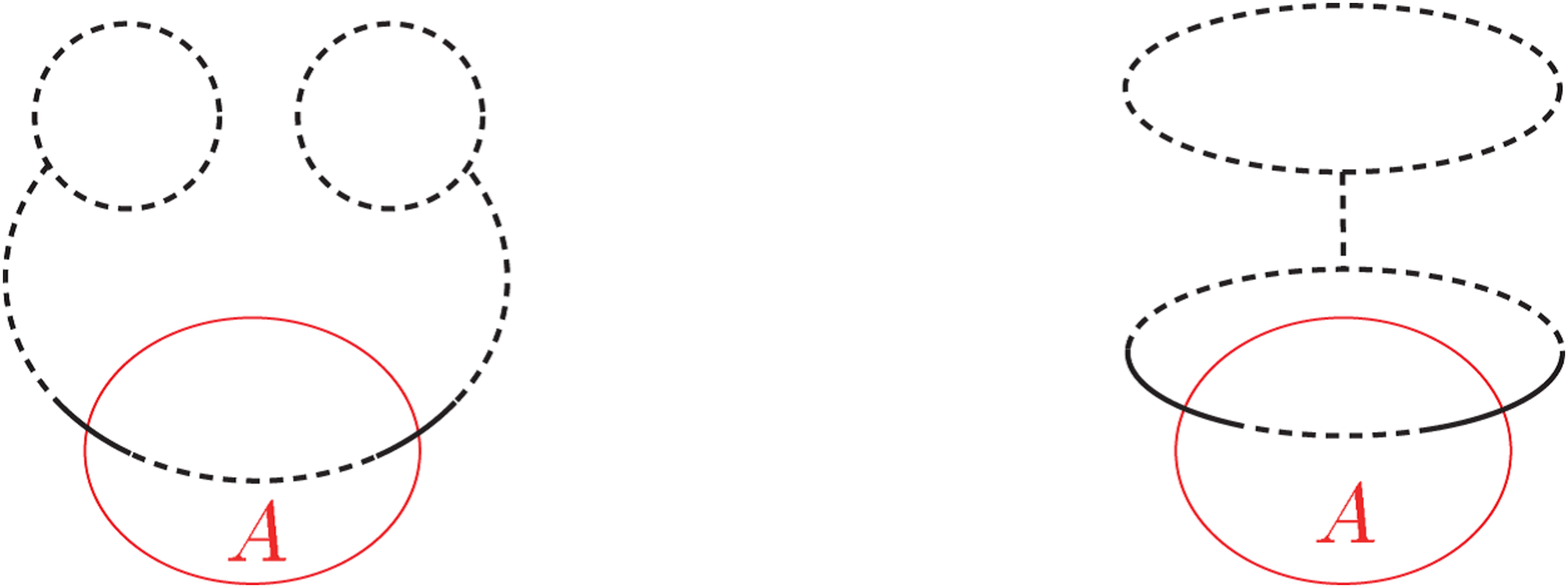}
\caption{Type 1 Annuli.}
\label{type1_annulus}
\end{figure}

\item[{\rm Type 2 (Hopf tangle type)}]
Let $\Gamma \subset S^3$ be a spatial handcuff-graph.  
Assume that one of the two loops of 
$\Gamma$ is a trivial knot bounding a disk 
$D$ such that 
$\Int \thinspace D$ intersects $\Gamma$ in an edge $e$ once and transversely.  
Set $V = \Nbd(\Gamma)$ and $A = D \cap E(V)$. 
We call $A$ a {\it Type $2$ annulus} for 
the handlebody-knot $(S^3, V)$. 
See Figure \ref{type2_annulus}. 
\begin{figure}[!hbt]
\centering
\includegraphics[width=8cm,clip]{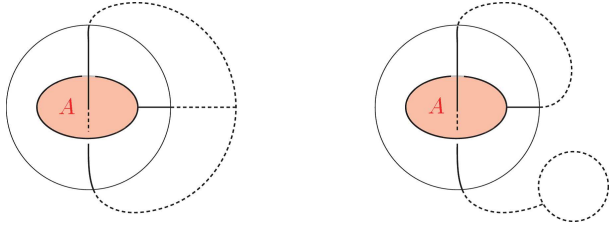}
\caption{Type 2 Annuli.}
\label{type2_annulus}
\end{figure}

\item[{\rm Type 3 (knot/link type)}]
Let $X$ be a solid torus embedded in $S^3$. 
Let $A$ be an annulus properly embedded in $E(X)$ such that 
$\partial A \cap \partial X$ consists of 
parallel non-trivial simple closed curves on 
$\partial X$. 
\begin{itemize}
\item
Let $\alpha$ be a properly embedded trivial simple arc  
in $X$ such that $\partial \alpha \cap \partial A = \emptyset$. 
Set $V = X \setminus \Int \thinspace \Nbd (\alpha)$. 
Then we call $A$ a {\it Type $3$-$1$ annulus} for the handlebody-knot $(S^3, V)$ provided that, 
if $\partial A$ bounds an essential disk in $X$, then any meridian disk of $X$ has non-empty 
intersection with $\alpha$. 
\item
Let $\partial A$ does not bound an essential disk in $X$. 
Let $\alpha$ be a properly embedded simple arc 
in $E(X)$ such that $\alpha \cap A = \emptyset$. 
Set $V = X \cup \Nbd (\alpha)$. 
Then we call $A$ a {\it Type $3$-$2$ annulus} for the handlebody-knot $(S^3, V)$ if 
$A$ is not parallel to the boundary of $V$. 
\end{itemize}

Let $X_1$, $X_2$ be two disjoint solid tori embedded in $S^3$. 
Assume that there exists an annulus $A$ properly embedded in 
$E(X_1 \sqcup X_2)$ such that 
$A \cap \partial X_i$ is a non-trivial simple closed curve in 
$\partial X_i$ for $i=1,2$. 
Let $e \subset E(X_1 \sqcup X_2) \setminus A$ be 
a proper arc connecting 
$\partial X_1$ and $\partial X_2$. 
Set $V = X_1 \cup X_2 \cup \Nbd (e)$. 
Then we call $A$ a {\it Type $3$-$3$ annulus} for the handlebody-knot $(S^3, V)$. 
A proper annulus $A$ in the exterior of a genus two handlebody-knot 
is said to be a {\it Type $3$ annulus} if it is 
a Type $3$-$1$, $3$-$2$ or $3$-$3$ annulus. 
Figure \ref{type3_annulus} shows 
schismatic pictures of Type 3 annuli. 
\begin{figure}[!hbt]
\centering
\includegraphics[width=12cm,clip]{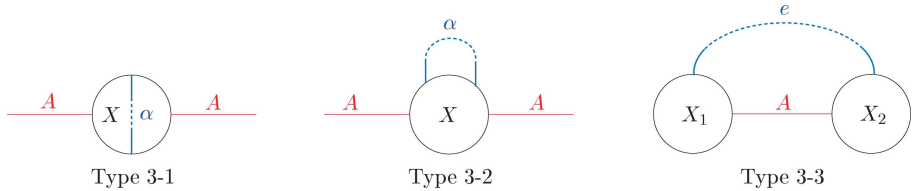}
\caption{Type 3 annuli.}
\label{type3_annulus}
\end{figure}

\item[{\rm Type 4 (Eudave-Mu\~noz type)}]
Let $K$ be an Eudave-Mu\~noz knot 
and let $P$ be an incompressible twice-punctured 
torus properly embedded 
in $E(K)$ so that $\partial P$ consists of 
the two parallel toroidal slopes of $K$. 
By Lemma \ref{lem:handlebodies and Eudave-Munoz knots}, 
$P$ cuts off $E(V)$ into two handlebodies of genus two. 
Let $V$ be one of them and set $A = \partial \Nbd(K) \setminus 
\Int \thinspace (\partial \Nbd(K) \cap \partial V)$. 
\begin{itemize}
\item
We call $A$ a {\it Type $4$-$1$ annulus} 
for the handlebody-knot $(S^3, V)$. 
\item
Let $U \subset S^3$ be a knot or a two component link 
contained in $E(V \cup N(K))$ such that 
$E(V \cup N(K) \cup U)$ 
is a compression body for $E(V \cup X)$. 
Let $i : E(U) \to S^3$ be an re-embedding 
such that $E(i(E(U)))$ is not a solid torus or two solid tori. 
Then we call $i(A)$ 
a {\it Type $4$-$2$ annulus} 
for the handlebody-knot. 
\end{itemize}
A proper annulus $A$ in the exterior of a genus two handlebody-knot 
is said to be a {\it Type $4$ annulus} if it is 
a Type $4$-$1$ or $4$-$2$ annulus. 
Figure \ref{type4_annulus} depicts schismatic pucture of 
an essential annulus of Type 4. 
\begin{figure}[!hbt]
\centering
\includegraphics[width=5cm,clip]{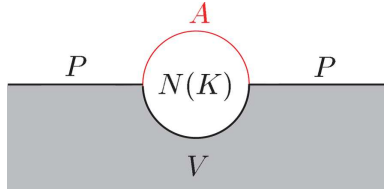}
\caption{A Type 4 annulus.}
\label{type4_annulus}
\end{figure}
\end{description}

\begin{remark}
The annuli listed above are not always essential. 
However, if $A \subset E(V)$ is an annulus of one of the above four types, at least we have 
the following by definition. 
\begin{itemize}
\item
each component of $\partial A$ is essential on $\partial V$. 
\item
$A$ is not parallel to the boundary of $V$. 
\end{itemize}
In Corollary \ref{cor:complete classification of essential annuli}, 
we will prove that if $(S^3, V)$ is irreducible, 
then the above annuli are actually essential. 
\end{remark}

\begin{example}
Figure \ref{example_annulus_1} 
shows several types of  
essential annuli in the exteriors of genus two 
handlebody-knots. 
\begin{figure}[!hbt]
\centering
\includegraphics[width=13.5cm,clip]{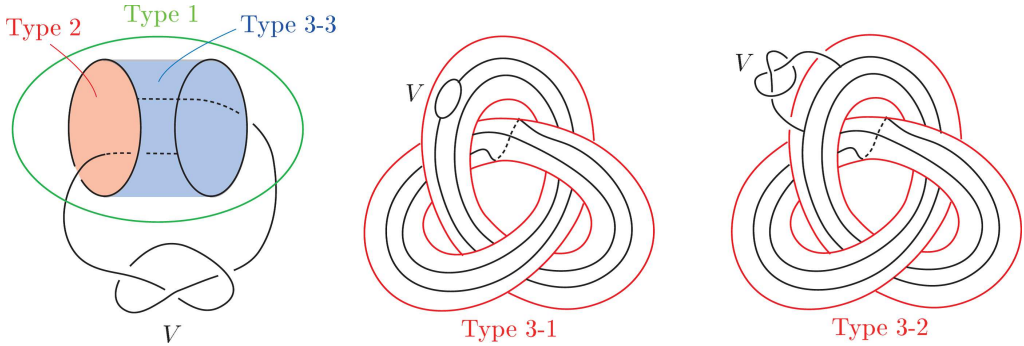}
\caption{Essential annuli.}
\label{example_annulus_1}
\end{figure}

\end{example}

Now we are ready to state the classification theorem of 
the essential annuli in the exterior of genus two handlebody-knots. 
This should be contrasted with Lemma \ref{lem:BZ85}. 

\begin{theorem}
\label{thm:classification of essential annuli}
Let $(S^3,V)$ be a genus two handlebody-knot. 
Then each essential annulus in the exterior of $V$ 
belongs to exactly one of the four Types 
listed above. 
\end{theorem}

Let $(S^3, V)$ be a genus two handlebody-knot. 
Let $A$ be an essential annulus $A$ in the exterior $E(V)$. 
Set $\partial A = a_1 \sqcup a_2$. 
We classify the configurations of the boundary of 
$A$ on $\partial V$ into the following four cases: 
\begin{description}
\item[{\rm Case} $1$] 
$a_1$ and $a_2$ are non-parallel, non-separating simple closed curves 
on $\partial V$.
\item[{\rm Case} $2$] 
$a_1$ is non-separating and $a_2$ is separating on $\partial V$.
\item[{\rm Case} $3$] 
$a_1$ and $a_2$ are parallel separating simple closed curves on $\partial V$.
\item[{\rm Case} $4$] 
$a_1$ and $a_2$ are parallel non-separating simple closed curves on $\partial V$.
\end{description}
\begin{figure}[!hbt]
\centering
\includegraphics[width=14.5cm,clip]{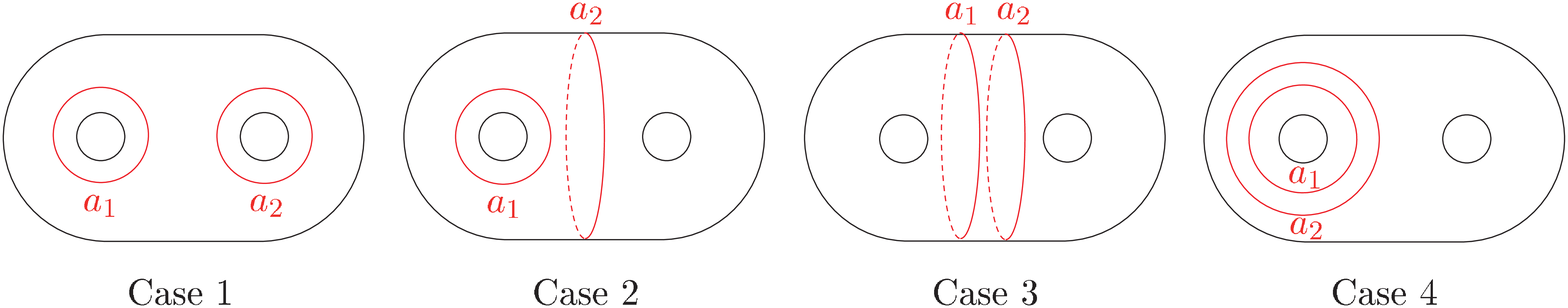}
\caption{}
\label{types_of_boundaries_of_annuli}
\end{figure}

We note that, by Lemma \ref{lem:essential surfaces in a handlebody}, 
the trivial handlebody-knot 
does not contain essential annuli in its exterior. 

\begin{lemma}
\label{lem:characterization of Type 1 and 2}
Let $(S^3, V)$ be a genus two handlebody-knot. 
Let $A \subset E(V)$ be an essential annulus. 
\begin{enumerate}
\item
If both $a_1$ and $a_2$ bound disks in $V$, then 
$A$ is a Type $1$ annulus. 
\item
If exactly one of $a_1$ and $a_2$ bounds a disk in $V$, then 
$A$ is a Type $2$ annulus. 
\end{enumerate}
\end{lemma}
\begin{proof}
(1) is straightforward from the definition. 
Let exactly one of $a_1$ and $a_2$, say $a_1$, 
bound a disk $E$ in $V$. 

If $E$ is non-separating in $V$, 
then we may assume that $a_2$ is an essential simple 
closed curve on the boundary of the solid torus 
$X = V \setminus \Int \thinspace \Nbd (E)$. 
Then the disk $A \cup E$ determines a Seifert surface of 
the core $K$ of $X$. 
It follows that $K$ is the trivial knot. 
Now, there is a handcuff-spine of $V$ consisting of 
two loops $e_1$, $e_2$ and one cut edge $e$ such that 
$e_1$ intersects $D$ once and transversely, $e_2 = K$ and 
$e \cap E = \emptyset$. 
This implies that $A$ is a Type 2 annulus. 

If $E$ is separating in $V$, 
then $V \setminus \Int \thinspace \Nbd (E)$ consists of two solid tori 
$X_1$ and $X_2$, and 
$a_2$ is an essential simple 
closed curve on the boundary of one of them, say $X_1$. 
Then, again, the disk $A \cup E$ determines a Seifert surface of 
the core $K$ of $X_1$. 
It follows that $K_1$ is the trivial knot. 
Fix meridian disks $E_1$ and $E_2$ of $X_1$ and $X_2$, respectively. 
There is a handcuff-spine of $V$ consisting of 
two loops $e_1$, $e_2$ and one cut edge $e$ such that 
$e_1 = K$, $e_2$ is the core of $X_2$, 
$e \cap (E_1 \cup E_2) = \emptyset$ and 
$e$ intersects $E$ once and transversely. 
This implies that $A$ is also a Type 2 annulus. 
\end{proof}

Let $P$ be a non-meridional, essential, planar surface 
properly embedded in the exterior of 
a knot $K$ in $S^3$. 
If $P$ is a disk, it is clear the $K$ is the trivial knot and 
$P$ is its Seifert surface. 
If $P$ is an annulus, then by Lemma \ref{lem:BZ85}, 
$K$ is a torus knot or a satellite knot and $P$ its cabling annulus. 
The next two lemmas, which plays an important role throughout 
this section, show that $P$ can be 
neither an $n$-punctured sphere for $n \geqslant 3$ odd 
nor a $4$-punctured sphere.

\begin{lemma}
\label{lem:odd-punctured spheres}
Let $P$ be a non-meridional 
planar surface with 
odd number of boundary components properly embedded 
in the exterior $E(K)$ of a knot $K$. 
Then $P$ is essential if and only if 
$K$ is the trivial knot and $P$ 
is a meridian disk 
of the solid torus $E(K)$. 
\end{lemma}
\begin{proof}
The sufficiency is clear. For necessity, let 
$F \subset E(K)$ be a non-meridional 
planar surface with odd number of boundary components. 
Then by capping off the boundary
components of $P$ by meridian disks of the filling solid torus, 
we obtain 
a non-separating sphere $\hat{P}$ in 
the $3$-manifold 
$S^3 (K; p/q)$ obtained from 
$S^3$ by performing the Dehn surgery along $K$ 
with the surgery slope $p/q$, 
where $p/q \neq 1/0$ 
is the boundary slope of $P$. 
Hence $S^3 (K; p/q)$ can be presented as $(S^2 \times S^1) \# M$. 
It follows that $H_1(S^3 (K; p/q)) \cong 
\Integer / p \Integer \cong \Integer \oplus H_1(M)$. 
This implies  that $p=0$ and $H_1(M) = 0$. 
By Corollary 8.3 of \cite{Gab87}, the 3-manifold $S^3 (K; 0)$ 
is prime and the genus of the knot is zero. 
Therefore $K$ is the trivial knot and $P$ is the meridian disk of 
$E(K)$. 
\end{proof}

\begin{remark}
It is proved in \cite{GL87} that if there exists a non-trivial knot that contains 
an essential planar surface $P$ of non-meridional boundary in its exterior, 
then the boundary-slope of $P$ is integral. 
\end{remark}

\begin{lemma}
\label{lem:4-punctured spheres}
The exterior of a knot in $S^3$ contains 
no properly embedded incompressible $4$-punctured 
sphere with integral boundary slope.  
\end{lemma}
The proof of Lemma \ref{lem:4-punctured spheres}, 
is given in the Appendix A by Cameron Gordon.

We remark that Lemmas \ref{lem:odd-punctured spheres} and 
\ref{lem:4-punctured spheres} are strongly related to 
the famous Cabling Conjecture, which was proposed 
Gonz\'alez-Acu\~na and Short. 
\begin{conjecture}
[The Cabling Conjecture \cite{GA86}] 
A Dehn surgery on a knot $K$ in $S^3$ can give a reducible manifold only when 
$K$ is a cable knot and the surgery slope is that of the cabling annulus. 
\end{conjecture}
The conjecture is known to hold for 
several classes of knots including 
satellite knots \cite{Sch90}, 
strongly invertible knots \cite{Eud92}, 
alternating knots \cite{MT93},  
symmetric knots \cite{LZ94, HS98} 
and the knots admitting bridge spheres with 
Hempel distance at least three \cite{Hof95, Hof98, BCJTT12}.  
However, the general case 
is still one of the most important open problems in 
the knot theory.
We note that if the exterior of every knot in $S^3$ contains 
no properly embedded essential planar surface of 
negative Euler characteristic with integral boundary slope, 
then the Cabling Conjecture is true.

\begin{lemma}[Classification of Case 1]
\label{lem:Classification of Case 1}
Let $A \subset E(V)$ be an essential annulus of Case $1$. 
Then $A$ is 
a Type $2$, $3$-$1$ or $3$-$3$ annulus.  
\end{lemma}
\begin{proof}
By Lemma \ref{lem:4-punctured spheres}, 
the 4-punctured sphere 
$P = \partial V \setminus \Int \thinspace \Nbd (a_1 \cup a_2)$ 
is compressible in $E(A)$. 
Let $D$ be a compressing disk for $P$. 

Assume first that $D$ lies in $V$. 
Let $D$ be separating in $V$. 
Then $V \setminus \Int \thinspace \Nbd(D)$ consists of two disjoint solid tori 
$X_1$ and $X_2$ such that 
$a_i \subset \partial X_i$ for $i=1,2$. 
If either $a_1$ or $a_2$, say $a_1$, is trivial on $\partial X_1$, 
$a_1$ is parallel to $\partial D$ on $\partial V$. 
This contradicts the assumption that $a_1$ is non-separating. 
Thus both $a_1$ and $a_2$ are 
non-trivial on $\partial X_1$ and $\partial X_2$, 
respectively. 
Then $A$ is a Type 3-3 annulus. 
Let $D$ be non-separating in $V$. 
If either $a_1$ or $a_2$ bounds a disk in $V$, 
it follows from 
Lemma \ref{lem:characterization of Type 1 and 2} 
that $A$ is a Type 2 annulus since 
$a_1$ and $a_2$ are not parallel on $\partial V$. 
Otherwise, $a_1$ and $a_2$ are parallel essential simple 
closed curves on the boundary of $X=V \setminus \Int \thinspace \Nbd (D; V)$. 
Since $a_1 \cup a_2$ separates $\partial E(X)$, 
$A$ is separating in $E(X)$. 
On the other hand, since $a_1$ and $a_2$ are not parallel on $\partial V$, 
each of the two annulus components of 
$\partial X \setminus \Int \thinspace \Nbd (\partial A; \partial X)$ 
meets $\partial \Nbd (D)$. 
It follows that $V \cap \Int \thinspace A \neq \emptyset$, 
whence a contradiction. 
See the left-hand side of Figure \ref{fig:proof_of_classification_of_case_1}. 
\begin{figure}[!hbt]
\centering
\includegraphics[width=10cm,clip]{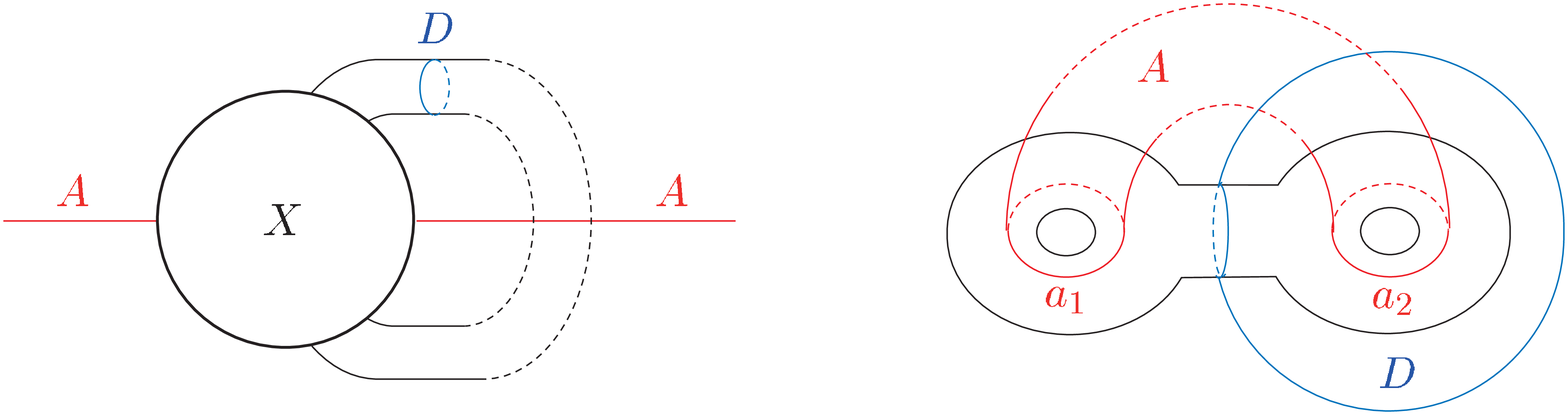}
\caption{}
\label{fig:proof_of_classification_of_case_1}
\end{figure}

Next, assume that $D$ lies in $E(V)$. 
Let $D$ be separating in $V$. 
Since $a_1$ and $a_2$ are non-parallel and 
non-separating on $\partial V$, 
each of the two components of 
$\partial V$ cut off by $\partial D$ contains $a_1$ or $a_2$. 
It follows that $D \cap A \neq \emptyset$, 
whence a contradiction. 
See the right-hand side of Figure \ref{fig:proof_of_classification_of_case_1}. 
Let $D$ be non-separating in $V$. 
Set $X= V \cup \Nbd(D)$. 
Since $\partial X$ is a torus in $S^3$, either $X$ or $E(X)$ 
is a solid torus. 
If $E(X)$ is a solid torus, then $V$ 
is the trivial genus two handlebody-knot. 
This contradicts 
Lemma \ref{lem:essential surfaces in a handlebody}. 
Hence $X$ is a solid torus. 
Since $A$ is essential in $E(V)$, 
neither $a_1$ nor $a_2$ is parallel to $\partial D$ on $\partial V$. 
It follows $a_1$ and $a_2$ are parallel essential simple 
closed curves on $\partial X$. 
Let $\alpha \subset X$ be the dual arc of $D$, that is,  
$\alpha$ is a simple arc properly embedded in $X$ such that 
$\Nbd(D) = \Nbd(\alpha)$. 
By \cite{Gor87}, $\alpha$ must be a trivial arc in 
$X$. 
This implies that $A$ is a Type 3-1 annulus. 
This completes the proof. 
\end{proof}

\begin{lemma}
\label{lem:intersection of an essential disk and an annulus}
Let $A$ be an essential annulus in $E(V)$. 
Suppose that $E(V)$ is boundary-reducible. 
Then there exists an essential disk $D$ in $E(V)$ 
such that $D \cap A = \emptyset$. 
\end{lemma} 
\begin{proof}
Let $D$ be an essential disk in $E(V)$. 
We minimize $\#(A \cap D)$ up to isotopy of $D$. 
If $A \cap D = \emptyset$, then we are done. 
Assume that $A \cap D \neq \emptyset$. 
Then a standard cut-and-paste argument allows us to retake 
an essential disk $D$ in $E(V)$ such that 
$A \cap D$ consists of essential circles or 
essential arcs. 
However, the existence of an essential circle in $A \cap D$ implies that 
$A$ is compressible, while the existence of an essential arc in $A \cap D$ 
implies that $A$ is boundary-compressible. 
This is a contradiction. 
\end{proof}

\begin{lemma}
\label{lem:there do not exists (2) nor (3)}
If $E(V)$ contains an essential annulus of Case $2$ or $3$, 
then $E(V)$ is boundary-irreducible. 
\end{lemma}
\begin{proof}
Let $E(V)$ be boundary-reducible and 
assume that there exists an essential annulus 
$A \subset E(V)$ be an essential annulus of Case $2$ or $3$. 
In what follows, 
we will prove that there exist an essential disk in $E(V)$ whose 
boundary is parallel to either $a_1$ or $a_2$ on $\partial V$. 
This implies that $A$ is compressible, whence a contradiction. 

By Lemma \ref{lem:intersection of an essential disk and an annulus}, 
there exists an essential disk $D$ in $E(V)$ disjoint from $A$. 

Assume that $D$ is separating in $E(V)$.   
since any mutually disjoint, separating, essential simple closed curves 
on a genus two closed surface are mutually parallel, 
$\partial D$ is parallel to $a_2$. 

Assume that $D$ is non-separating in $E(V)$. 
Suppose that $A$ is of Case $2$. 
Let $P_1$ and $P_2$ be the pair of pants component and 
the once-punctured component of $\partial V$ cut off by $\partial A$. 
If $\partial D$ is contained in $P_1$, $\partial D$ is parallel to 
$a_1$ on $\partial V$. 
If $\partial D$ is contained in $P_2$, then there exists 
a simple closed curve $l$ on $P_2$ that intersects 
$\partial D$ once and transversely. 
Then the closure $D'$ of 
$\partial \Nbd (D \cup l ; E(V)) \setminus \partial M$ 
is an essential separating disk in $E(V)$ disjoint from $A$. 
Then, by the above argument, $\partial D'$ is parallel to $a_2$ on $\partial V$. 
Suppose that $A$ is of Case $3$. 
Since $D$ is non-separating, 
$\partial D$ is contained in a once-punctured component of 
$\partial V$ cut off by $\partial A$. 
Then we obtain an essential disk 
$D'$ in $E(V)$ so that $\partial D$ is parallel to $a_2$ on $\partial V$ 
as above. 
\end{proof}

\begin{lemma}[Classification of Case 2]
\label{lem:Classification of Case 2}
Let $A \subset E(V)$ be an essential annulus of Case $2$. 
Then $A$ is a Type $2$ annulus. 
\end{lemma}
\begin{proof}
By Lemma \ref{lem:there do not exists (2) nor (3)}, 
We may assume that 
$E(V)$ is boundary-irreducible. 
Let $P$ be the component of 
$\partial V \setminus \Int \thinspace 
\Nbd (a_1 \cup a_2)$ that is homeomorphic to 
a pair of pants. 
Lemma \ref{lem:odd-punctured spheres} 
implies that $P$ is compressible in $E(A)$. 
Since $\partial V$ is incompressible in $E(V)$, is 
$P$ is compressible in $V \cap E(A)$. 
It follows that either $a_1$ or $a_2$ bounds a disk 
in $V$. 
By Lemma \ref{lem:characterization of Type 1 and 2}, 
$A$ is a Type 1 or 2 annulus. 
Since $a_1$ and $a_2$ are not parallel by assumption, 
it follows that $A$ is a Type 2 annulus. 
\end{proof}

\begin{lemma}[Classification of the Case 3]
\label{lem:Classification of Case 3}
Let $A \subset E(V)$ be an essential annulus of Case $3$. 
Then $A$ is a Type $1$ annulus.  
\end{lemma}
\begin{proof}
By Lemma \ref{lem:there do not exists (2) nor (3)}, 
we may assume that 
$E(V)$ is boundary-irreducible. 
Let $A' \subset \partial V$ be the annulus with 
$\partial A' = a_1 \sqcup a_2$. 
Then the torus $A \cup A'$ bounds a solid torus $X$ in $S^3$. 
Let $P$ and $Q$ be the once-punctured torus components of 
$\partial V \setminus \Int \thinspace A'$. 
Suppose first that $P \sqcup Q$ is contained in $X$. 
Then $P \sqcup Q$ is compressible in $X$ 
since a solid torus does not contain incompressible 
once-punctured tori. 
Since $\partial V$ is incompressible in $E(V)$, 
$P \sqcup Q$ is compressible in $V$. 
It follows that both $a_1$ and $a_2$ bound disks 
in $V$, which implies by Lemma \ref{lem:characterization of Type 1 and 2} 
that 
$A$ can be extended to a $2$-decomposing sphere of $V$. 
Suppose next that $P \sqcup Q$ is contained in $E(X)$. 
Since both $P$ and $Q$ 
determine Seifert surfaces of the core of $X$, 
both $\partial P$ and $\partial Q$ 
are parallel to the preferred-longitude of $X$. 
This implies that $A$ and $A'$ are parallel in $X$. 
However, this contradicts the assumption that $A$ is essential. 
\end{proof}

We recall the following theorem by Hayashi and Shimokawa, 
which will be needed in the proof of 
Lemma \ref{lem:satellite knots admitting 2-pounc. tori with non-integral boundary slope}. 
\begin{theorem}[\cite{HS98}]
\label{thm:HS98}
Let $Y$ be a solid torus and $K \subset Y$ 
be a non-cabled knot. 
Assume that $\partial Y$ is incompressible in 
$Y \setminus \Int \thinspace \Nbd (K)$. 
Let $Y (K; r)$ be the $3$-manifold obtained 
from $Y$ by performing the Dehn surgery along 
$K$ with the surgery slope $r$. 
If $Y (K; r)$ contains a separating essential annulus 
$\tilde{A}$ such that each component of $\partial \tilde{A}$ is primitive 
with respect to $Y$, then 
the slope $r$ is integral. 
\end{theorem}

\begin{lemma}
\label{lem:satellite knots admitting 2-pounc. tori with non-integral boundary slope}
Let $K$ be a knot in $S^3$. 
If there exists an incompressible twice-punctured torus $P$ in 
$E(K)$ with non-integral boundary slopes that 
cuts off $E(K)$ into two genus two handlebodies, 
then $K$ is a hyperbolic knot. 
\end{lemma}
\begin{proof}
It is clear that $K$ is neither the trivial knot nor 
a torus knot since 
it is well-known that these knots do not contain 
essential twice-punctured tori in their exteriors. 
Let $K$ be a satellite knot. 
Then there exists an essential torus in $E(K)$. 
Each essential torus $T$ cuts off $S^3$ into two components 
$Y_1$ and $Y_2$, where $Y_1$ is a solid torus. 
We remark that $K \subset Y_1$, otherwise $T$ is 
compressible in $E(K)$. 
Assume that 
$\#  ( P \cap T )$ is minimal up to isotopy of $T$. 
We note that $P \cap T \neq \emptyset$ since 
$P$ cuts off $E(V)$ into two handlebodies $V$ and $V'$. 
We also note that each component of 
$P \cap Y_2$ is essential 
since $P$ is essential and $\#(P \cap T)$ is minimal. 
Let $K_1$ be the core of the solid torus $Y_1$.

\begin{claim}
\label{claim:there is no boundary-parallel loop}
No component of $P \cap T$ is parallel to 
a component of $\partial P$ on $P$. 
\end{claim}
\noindent {\it Proof of Claim $\ref{claim:there is no boundary-parallel loop}$.}
Assume for contradiction that $P \cap T$ contains a 
simple closed curve $l$ parallel to $\partial P$ on $P$. 
Without loss of generality, we may assume that $l$ cuts off  
an annulus $P_0$ from $P$ so that 
$\Int \thinspace P_0 \cap T = \emptyset$. 
\begin{figure}[!hbt]
\centering
\includegraphics[width=3cm,clip]{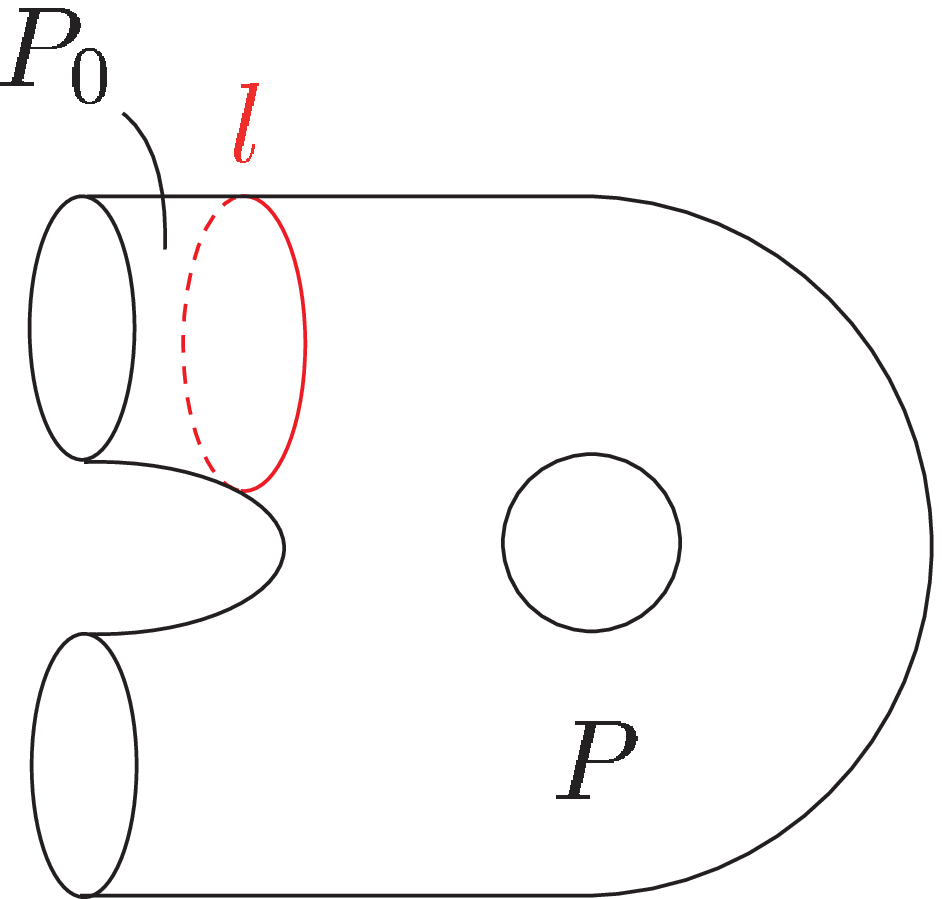}
\caption{}
\label{twice_punctured_torus1}
\end{figure}
See Figure \ref{twice_punctured_torus1}. 
Then $P_0$ is a peripherally compressing annulus for $T$. 
By Lemma \ref{lem:perhipherally compressing a torus}, 
we obtain by peripherally compressing $T$ along $P_0$ 
an essential annulus $T'$ in $E(K)$ 
with non-integral boundary-slope. 
This contradicts Lemma \ref{lem:BZ85}. 

\begin{claim}
\label{claim:mutually parallel loops}
The number of mutually parallel loops of $P \cap T$ on $P$ is 
at most two. 
\end{claim}
\noindent {\it Proof of Claim $\ref{claim:mutually parallel loops}$.}
Assume for contradiction that 
$P \cap T$ contains 
mutually parallel $n \geqslant 3$ loops on $P$. 
Then there exist annulus components 
$P_1 \subset P \cap Y_1$ and 
$P_2 \subset P \cap Y_2$. 
Recall that $P_2$ is essential in $Y_2$. 
By Lemma \ref{lem:BZ85}, 
$P_2$ is a cabling annulus for $K_1$, or 
$P_2$ can be extended to a decomposing sphere for $K_1$. 

In the former case, the slopes $P \cap T$ are 
integral with respect to the meridian and preferred longitude 
of $K_1$. 
Hence $P_1$ is parallel to $\partial Y_1$ from both side. 
This implies that we can reduce the number of components of 
$P \cap T$, whence a contradiction. 

In the latter case, the slopes $P \cap T$ bound 
meridian disks in $Y_1$. 
By Claim \ref{claim:there is no boundary-parallel loop}, 
each component of $P$ cut off by $P \cap T$ 
is either an annulus, a pair of pants, 
a 4-punctured sphere or a once-punctured torus. 
We see that $P \cap Y_2$ consists of only essential annuli 
as follows. 
Let $Q$ be a component of $P \cap Y_2$. 
Since $P \cap \partial Y_2$ is meridional in $Y_2$, 
$\# \partial Q$ is even, otherwise 
$S^3$ contains a non-separating sphere or torus, which 
is a contradiction. 
Thus $Q$ is neither a pair of pants nor a once-punctured torus. 
On the other hand, by Claim \ref{claim:there is no boundary-parallel loop}, 
a $4$-punctured sphere component of $P$ cut off by $P \cap T$ (if any) 
lies in $Y_1$. 
Thus $Q$ is not a 4-punctured sphere. 
As a consequence, $Q$ is an annulus. 
Among the essential annuli $P \cap Y_2$, 
take an outermost one $P'_2$ in $Y_2$. 
By tubing $P'_2$ along a sub-annulus on $T$ 
whose interior does not intersect $P$, 
we obtain an essential torus $T'$ in $E(K)$ with $P \cap T' = \emptyset$. 
See Figure \ref{fig:tubing_decomposing_annulus}. 
\begin{figure}[!hbt]
\centering
\includegraphics[width=4.5cm,clip]{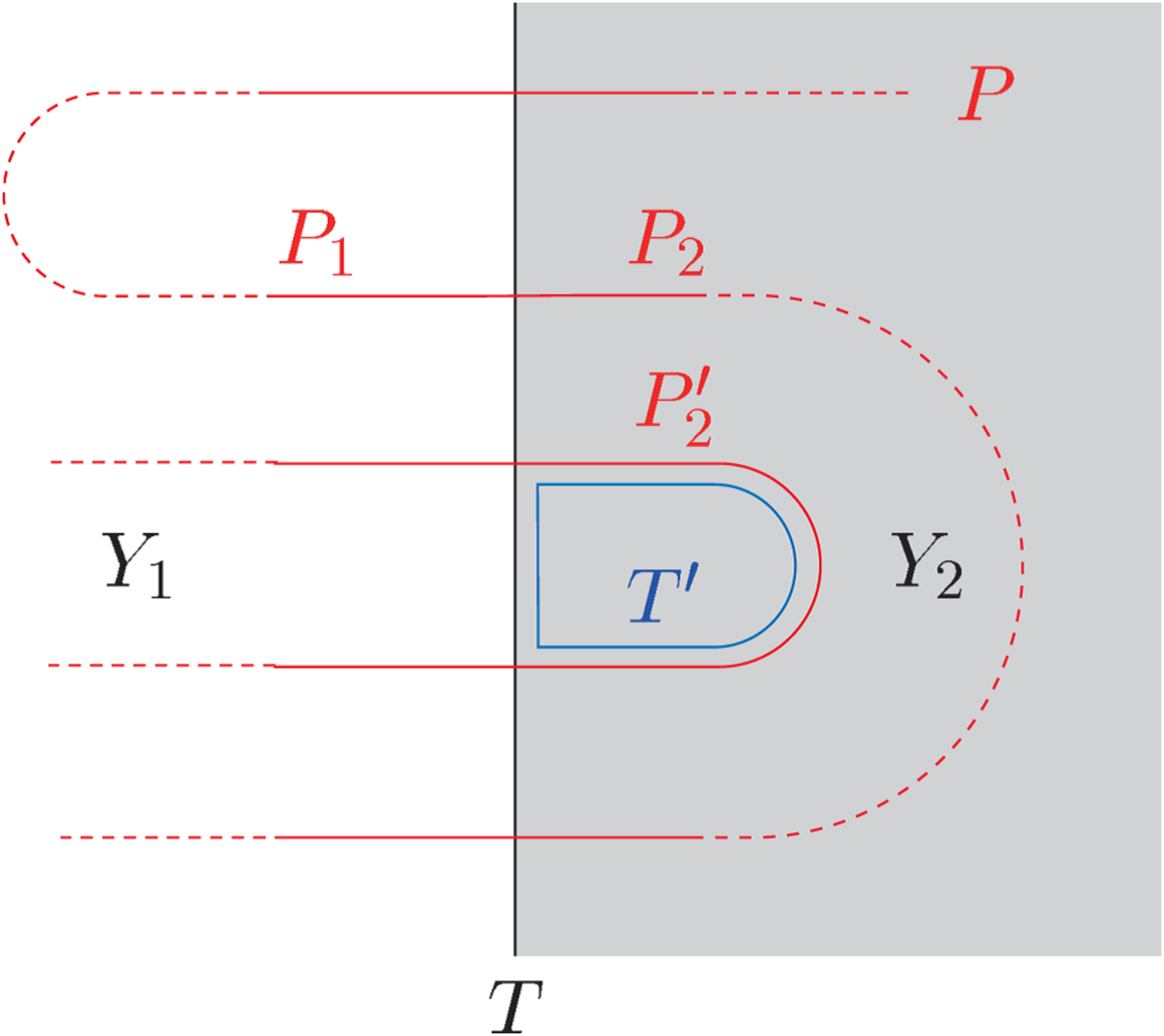}
\caption{}
\label{fig:tubing_decomposing_annulus}
\end{figure}
This implies $T' \subset V$ or  $T' \subset V'$. 
Then we have $T' \subset V$ or $T' \subset V'$. 
This contradicts Lemma \ref{lem:essential surfaces in a handlebody}. 

\begin{claim}
\label{claim:there is no separating loop}
$P \cap T$ does not contain separating simple closed curves on $P$. 
\end{claim}
\noindent {\it Proof of Claim $\ref{claim:there is no separating loop}$.}
Assume for contradiction that $P \cap T$ contains a separating 
simple closed curve $l$ on $P$. 
By Claim \ref{claim:there is no boundary-parallel loop}, 
$l$ is parallel to no component of $\partial P$. 
\begin{figure}[!hbt]
\centering
\includegraphics[width=3cm,clip]{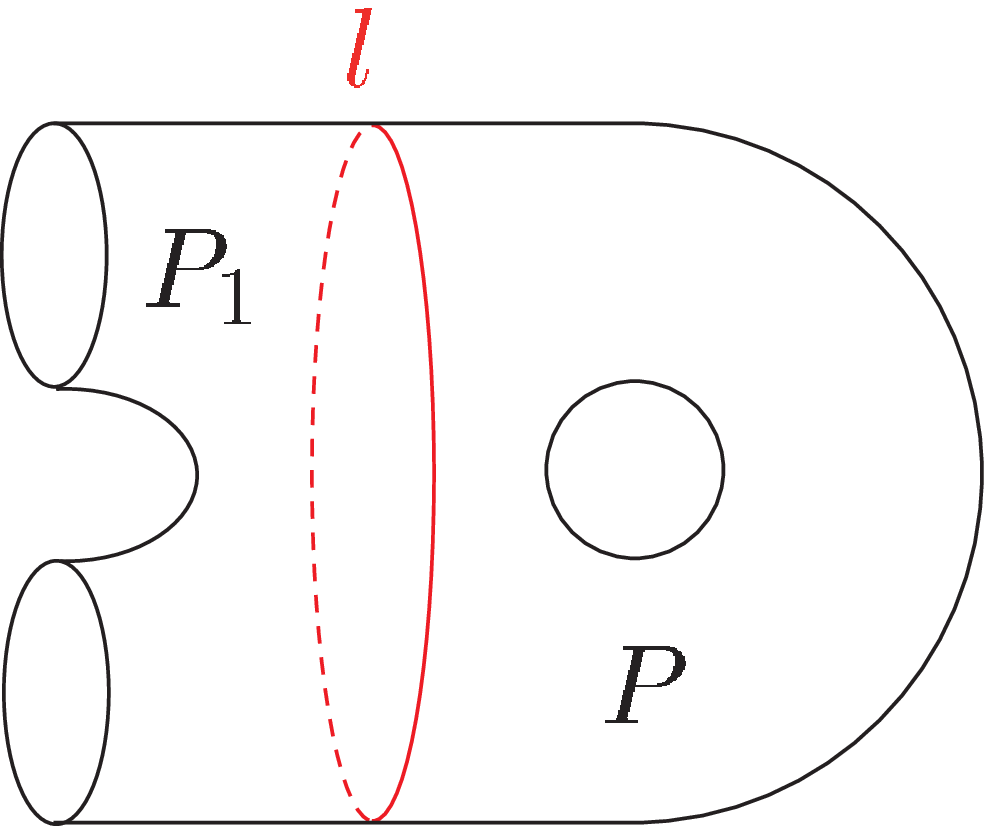}
\caption{}
\label{twice_punctured_torus2}
\end{figure}
If there exist components of $P \cap T$ 
which are not parallel to $l$ on $P$, 
they must be mutually parallel 
non-separating simple closed curves. 
Moreover, since $T$ is separating in $E(K)$, 
the number of such components are exactly two by 
Lemma \ref{claim:mutually parallel loops}. 
 Let $m$ be the number of components of $P \cap T$ parallel 
to $l$ on $P$. 
Let $n$ be the number of the non-separating components of $P \cap T$. 
Then $(m,n)$ is $(1,0)$, $(2,0)$, $(1,2)$ or $(2,2)$. 
Let $P_1$ be the pair of pants component of 
$P$ cut off by $P \cap T$ 
such that $\partial P \subset \partial P_1$. 
See the right-hand side of Figure \ref{twice_punctured_torus2}. 
When $(m,n) = (1,0)$, let $P_2$ be 
the once-punctured torus component of 
$P \setminus \Int \thinspace \Nbd(P \cap T; P)$. 
Then $P_2$ is a Seifert surface of the core 
$K_1$ of $Y_1$. 
In particular, the slope $l$ is the preferred longitude of $Y_1$. 
Hence there is a re-embedding  
$h: Y_1 \to S^3$ such that $h(l)$ bounds a disk in $E(h(Y_1))$. 
Then by adding a disk, $h (P_1)$ 
can be extended to a proper annulus $\hat{A}$ 
in $E(h(K))$ with non-integral boundary slope. 
It follows that $\hat{A}$ is parallel to the boundary of $E(h(K))$. 
However, $\hat{A}$ must be non-separating in $E(h(K))$ since 
$\hat{A}$ intersects $E (h(Y_1))$ 
in a single meridional disk. 
This is a contradiction. 
When $(m,n) = (1,2)$, $P \cap Y_2$ contains a component 
which is an essential pair of pants in $Y_2 = E(Y_1)$. 
This contradicts Lemma \ref{lem:odd-punctured spheres}. 
When $(m,n) = (2,0)$ or $(2,2)$, by Lemma \ref{lem:BZ85}, 
the boundary-slope of $P \cap T$ is 
cabling or meridional for $Y_1$. 
In the former case, 
we also have a contradiction by a similar argument of the case $(m,n) = (1,0)$. 
In the later case, there is a component of $P \cap Y_2$ that 
can be extended to a decomposing sphere for $K_1$. 
Then by the same argument of the last part of the proof of 
Claim \ref{claim:mutually parallel loops}, 
there exists an essential torus in $E(K)$ which 
does not intersect $P$. This is a contradiction.  

\begin{claim}
\label{claim:there exist only two parallel non-separating loop}
$T \cap P$ consists of two parallel non-separating 
simple closed curves on $P$. 
\end{claim}
\noindent {\it Proof of Claim $\ref{claim:there exist only two parallel non-separating loop}$.}
By Claims \ref{claim:mutually parallel loops} and 
\ref{claim:there is no separating loop}, 
$P \cap T$ consists of two parallel non-separating 
simple closed curves on $P$ 
(see the left-hand side of Figure \ref{twice_punctured_torus3}), or 
four non-separating 
simple closed curves on $P$ such that the two of them are parallel 
and the remaining two are also parallel 
(see the right-hand side of Figure \ref{twice_punctured_torus3}).  
\begin{figure}[!hbt]
\centering
\includegraphics[width=8cm,clip]{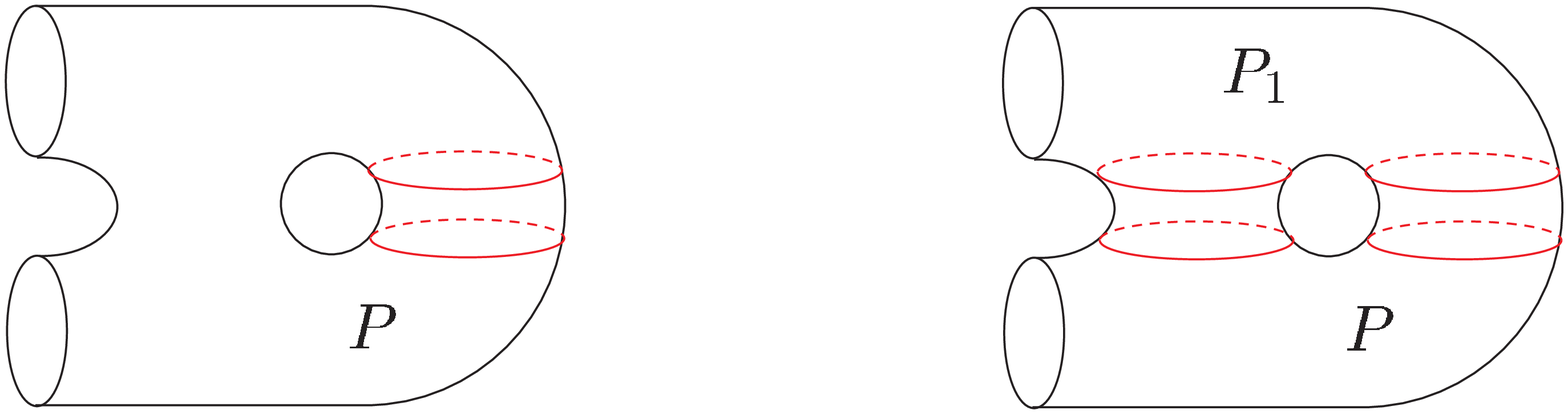}
\caption{}
\label{twice_punctured_torus3}
\end{figure}
In the latter case, let $P_1$ be one of the two pairs of pants 
of $P$ cut off by $P \cap T$. 
Since $P \cap T$ consists of mutually parallel integral slope on 
$T$ with respect to the knot $K_1$, 
we can re-embed $Y_1$ by a map $h: Y_1 \to S^3$ so that 
each component of $h(P \cap T)$ bounds a disk in $E(h(Y_1))$. 
Then by adding disks to $h(P_1)$ 
along the boundary circles $h(\partial P_1 \setminus \partial N(K))$, we obtain a disk 
whose boundary is not integral with respect to $h(K)$. 
This is a contradiction.  

\begin{claim}
\label{claim:essential torus and cabling annulus}
$K_1$ is a torus knot and $Y_2 \cap T$ is the cabling annulus. 
\end{claim}
\noindent {\it Proof of Claim $\ref{claim:essential torus and cabling annulus}$.} 
By Claim 
\ref{claim:there exist only two parallel non-separating loop}, 
both $V \cap T$ and $V' \cap T$ 
are separating incompressible annuli 
in the handlebodies $V$ and $V'$, respectively. 
Then it follows from the classification of 
esential separating annuli in a genus two 
handlebody \cite{Kob84}, 
both $V \cap Y_2$ and $V' \cap Y_2$ 
are solid tori. 
This fact and Lemma \ref{lem:BZ85} imply that 
$K_1$ is a torus knot and $P \cap Y_2$ is its cabling annulus. 

\begin{claim}
\label{claim:there exist no essential tori}
There exist no essential tori in $Y_1 \setminus \Int \Nbd(K)$. 
\end{claim}
\noindent {\it Proof of Claim $\ref{claim:there exist no essential tori}$.}
Assume for contradiction that $Y_1 \setminus \Int \Nbd(K)$ contains 
an essential torus $T'$. 
We also assume that $\#(P \cap T')$ is minimal up to isotopy in 
$Y_1 \setminus \Int \Nbd(K)$. 
Clearly, $T'$ is also essential in $E(K)$ and $T'$ cuts $S^3$ 
into two components $Y'_1$ and $Y'_2$, where $Y'_1$ is a solid torus. 
Then by Claim \ref{claim:essential torus and cabling annulus}, 
$Y'_2$ is also a torus knot exterior. 
We note that $Y_2$ and $Y'_2$ are disjoint, otherwise 
$T'$ is parallel to $T$ in $E(K)$. 
Since $P \cap Y'_2$ is essential in $Y'_2$, $P \cap Y'_2$ is 
a non-empty disjoint union of the cabling annuli in $Y'_2$. 
Let $\gamma$ be the core of the annulus $P \cap Y_2$ and 
let $\gamma_1, \gamma_2, \ldots, \gamma_n$ be the cores of the annuli of $P \cap Y'_2$. 
Since $Y_2$ and $Y'_2$ are disjoint, we may assume (up to isotopy) that 
$\gamma \cap (\bigcup_{i=1}^n \gamma_i ) = \emptyset$, $P \cap Y_2 = \Nbd (\gamma ; P)$ and 
$P \cap Y'_2 = \Nbd ( \bigcup_{i=1}^n \gamma_i ; P)$. 
By the same argument in the proof of Claim \ref{claim:there is no separating loop}, 
none of $\gamma_1, \gamma_2, \ldots, \gamma_n$ is separating in $P$. 
Assume that a component one of the circles $\gamma_1, \gamma_2, \ldots, \gamma_n$, 
say $\gamma_1$, is parallel to $\gamma$ on $P$. 
Let $T_1 = T \cap V_1$ and let $T'_1$ be a component of $T' \cap V_1$ such that 
$\partial T'_1 =  \partial \Nbd (\gamma_1 ; P)$. 
We remark that $T_1$ and $T'_1$ are separating incompressible annuli in $V_1$ and 
all components of $\partial T_1$ and $\partial T'_1$ are parallel on $P$.  
Then by \cite{Kob84}, $T_1$ must be contained in the solid torus component of 
$V_1$ cut off by $T'_1$. 
This is impossible since $Y_2$ and $Y'_2$ are disjoint. 
Therefore, $P$ cut off by $T \cup T'$ contains a pair of pants component $P_1$ 
exactly one of whose boundary components lies on $\partial N(K)$. 
Now, as in the proof of Claim \ref{claim:there exist only two parallel non-separating loop}, 
we can re-embed $Y_1 \cap Y'_1$ by a map $h: Y_1 \cap Y'_1 \to S^3$ so that 
each component of $h(P \cap (T \cup T'))$ bounds a disk in $E(h(Y_1 \cap Y'_1))$. 
Then by adding disks to $h(P_1)$ 
along the boundary circles $h(\partial P_1 \setminus \partial N(K))$, we obtain a disk 
whose boundary is not integral with respect to $h(K)$. 
This is a contradiction.

\vspace{1em}
We set $M = Y_1 \setminus \Int \thinspace N(K)$. 
By Claim \ref{claim:there exist only two parallel non-separating loop}, 
$P \cap M$ is a essential separating 4-punctured sphere in $M$. 
Then $P \cap M$ is 
naturally extend to a separating essential annulus 
$\tilde{A}$ in the 3-manifold $Y_1(K ; r)$ obtained 
from $Y_1$ by performing the Dehn surgery along 
$K$ with the surgery slope $r$ defined by 
the boundary slope of $P$ on $\partial N(K)$. 
See Figure \ref{type4_annulus_2}. 
\begin{figure}[!hbt]
\centering
\includegraphics[width=13.5cm,clip]{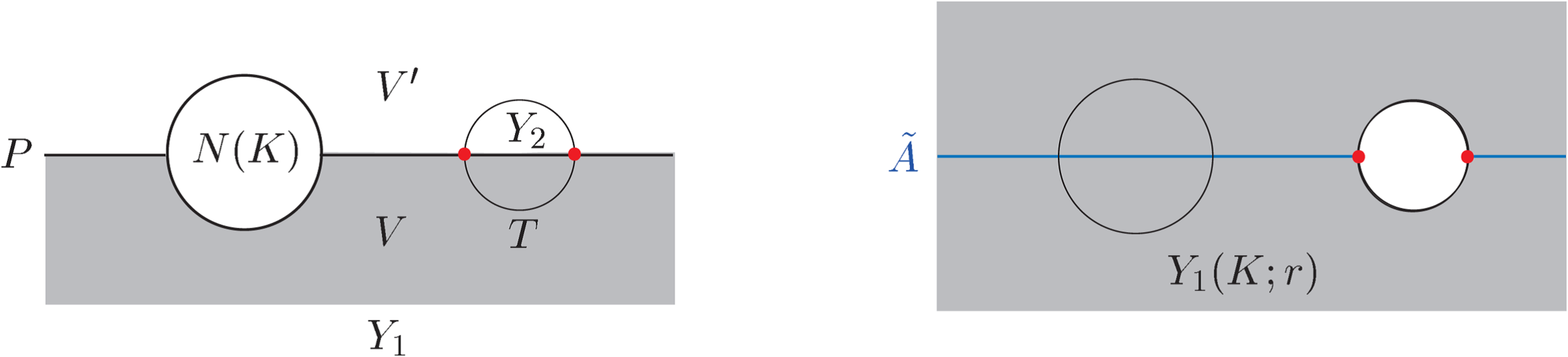}
\caption{}
\label{type4_annulus_2}
\end{figure}
Since $P \cap Y_2$ is the cabling annulus 
of the core $K_1$ of $Y_1$, 
each component of $\partial \tilde{A} = P \cap \partial Y_1$ is 
primitive with respect to the solid torus $Y_1$. 
By definition $\partial Y_1 = T$ is incompressible in 
$M$. 
Also, by Claim \ref{claim:there exist no essential tori}, 
$M$ does not 
contain essential tori. 
Hence if $K$ can not be isotoped onto $T$, 
it follows from Theorem \ref{thm:HS98} that the slope $r$ is integral. 
This is a contradiction. 
Otherwise, $M$ is 
a Seifert fiber space so-called a {\it cabling space}, and 
$P \cap M$ is an essential 4-punctured sphere in it. 
However, this is impossible by Lemma 3.1 in \cite{GL84}.  
\end{proof}

The following theorem by Przytycki describes the incompressibility of surfaces 
before and after performing the Dehn filling.  
\begin{theorem}[\cite{Prz83}]
\label{thm: Theorem of Przytycki}
Let $M$ be a compact $3$-manifold whose boundary is a single torus. 
Let $P$ be a compact orientable surface properly embedded in $M$ such that 
\begin{enumerate}
\item
$P$ cuts off $M$ into two handlebodies; 
\item
$\partial P$ consists of two non-trivial simple closed curves on $\partial M$; and 
\item
$P$ is not parallel to $\partial M$. 
\end{enumerate}
Let $\hat{M}$ be the $3$-manifold obtained from $M$ by performing the Dehn filling 
along the boundary slope of $P$. 
Let $\hat{P}$ be the surface in $\hat{M}$ naturally obtained by capping off the boundary of $P$. 
Then $P$ is incompressible in $M$ if and only if $\hat{P}$ is incompressible in $\hat{M}$. 
\end{theorem}

Theorems \ref{thm:classification of hyperbolic knots admitting a non-integral toroidal surgery} 
and \ref{thm: Theorem of Przytycki} together with 
Lemmas \ref{lem:handlebodies and Eudave-Munoz knots} and 
\ref{lem:satellite knots admitting 2-pounc. tori with non-integral boundary slope} 
provide the following corollary, which plays a key role 
for the classification of the essential annuli of Case 4. 

\begin{corollary}
\label{cor:essential twice-punctuerd tori and Eudve-Munoz knots}
Let $K$ be a knot in $S^3$. 
Let $P$ be a compact orientable twice-punctured torus properly embedded in $E(K)$. 
Then $P$ is an essential surface that cuts off $E(K)$ into two genus two handlebodies 
if and only if 
$K$ is an Eudave-Mu\~noz knot 
and $P$ in an incompressible twice-punctured torus 
properly embedded in $E(K)$ such that $\partial P$ consists of 
the two parallel toroidal slopes of $K$. 
\end{corollary}

\begin{lemma}[Classification of Case 4]
\label{lem:Classification of Case 4}
Let $A \subset E(V)$ be an essential annulus of Case $4$. 
Then $A$ is 
a Type $1$, $3$-$1$, $3$-$2$ or $4$ annulus.  
\end{lemma}
\begin{proof}
If both $a_1$ and $a_2$ bounds a disk in $V$, 
$A$ is a Type 1 annulus by Lemma \ref{lem:characterization of Type 1 and 2}.  
In the following, we assume that both 
$a_1$ and $a_2$ do not bound disks in $V$.  
Let $A' \subset \partial V$ be the annulus with 
$\partial A' = a_1 \sqcup a_2$. 
Then the torus $A \cup A'$ bounds a solid torus $X$ in $S^3$. 
Set $P = \partial V \setminus \Int \thinspace A'$. 

Suppose first that $P$ is contained in $X$. 
Then there is a compressing disk $D$ for $P$. 

Suppose that $\partial D$ is 
non-separating on $P$ and 
let $P' \subset X$ be the annulus obtained by compressing 
$P$ along $D$. 
\begin{claim}
\label{claim:annulus in a solid torus}
$A' \cup P'$ bounds a solid torus in $X$. 
\end{claim} 
\noindent {\it Proof of Claim $\ref{claim:annulus in a solid torus}$.}
When $P'$ is incompressible in $X$, 
$P'$ is parallel to either $A$ or $A'$. 
In each case, it is clear that 
$A' \cup P'$ bounds a solid torus in $X$. 
Suppose that $P$ is compressible in $X$. 
Then by compressing $P'$, 
we obtain two disks $D_1$ and $D_2$ 
bounded by $a_1$ and $a_2$, respectively. 
The disks $D_1$ and $D_2$ cuts off $X$ into two 3-balls $B_1$ and $B_2$ 
such that $A \subset \partial B_1$ and $A' \subset \partial B_2$. 
Now, $P'$ is obtained by tubing $D_1$ and $D_2$ along a simple 
arc $\gamma$ connecting $D_1$ and $D_2$. 
If $\gamma \subset B_1$, we are done 
(see the left-hand side of Figure \ref{non-separating_and_compressing}). 
\begin{figure}[!hbt]
\centering
\includegraphics[width=9cm,clip]{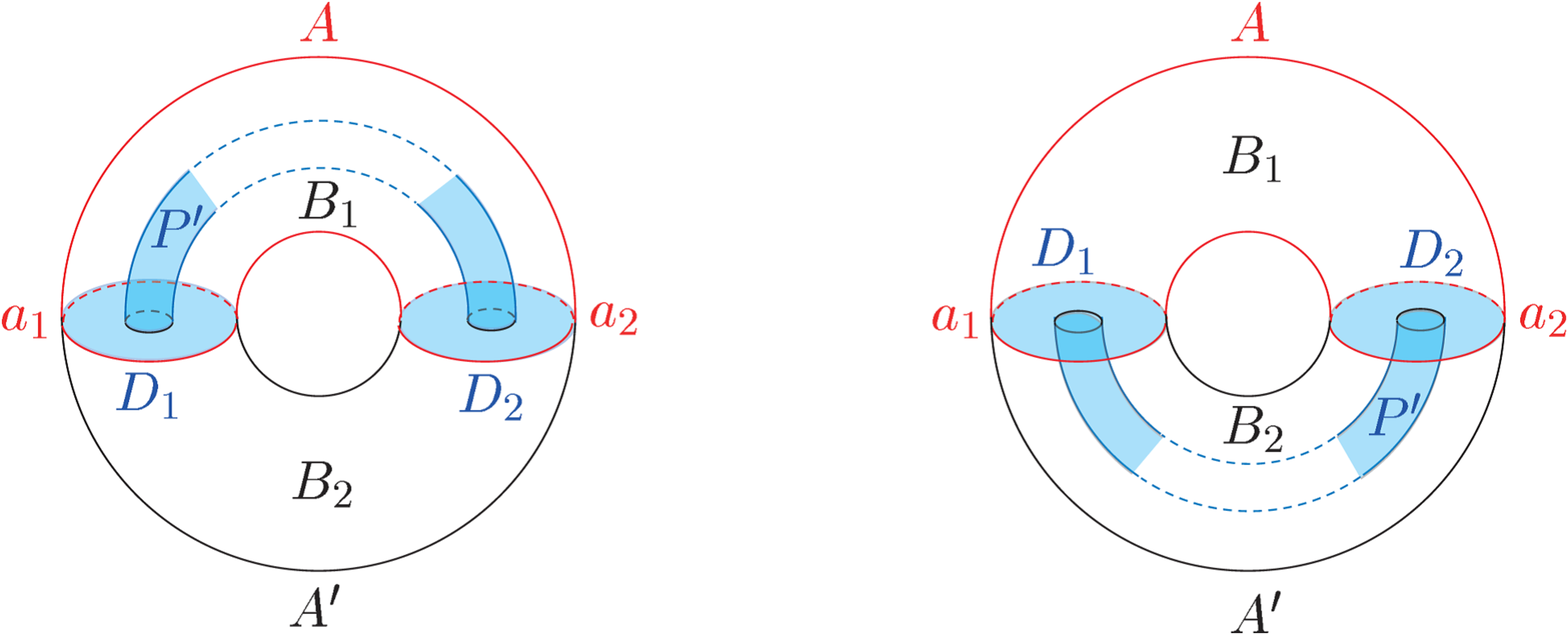}
\caption{}
\label{non-separating_and_compressing}
\end{figure}
Assume that $\gamma \subset B_2$. 
Then 
$A' \cup P'$ is bounding $V'=B_2 \setminus \Int \thinspace \Nbd(\gamma; B_2)$. 
The handlebody $V$ is obtained from $V'$ by attaching a 1-handle or drilling along a 
simple arc 
(see the right-hand side of Figure \ref{non-separating_and_compressing}). 
The latter is impossible since $A$ is incompressible in $X$. 
This implies that $V'$ is also a solid torus, 
whence the claim.  

\vspace{1em}

By Claim \ref{claim:annulus in a solid torus}, 
$A' \cup P'$ bounds a solid torus $X'$ in $X$. 
Since $\partial A' = \partial P' = a_1 \cup a_2$, 
$a_1$ and $a_2$ are parallel non-trivial simple closed curves on 
$\partial X'$. 
If $D \subset E(V)$, then $V$ is obtained from $X'$ 
by drilling $X'$ along a properly embedded simple arc $\alpha$ 
in $X'$ such that $\partial \alpha \cap \partial A = \emptyset$. 
By \cite{Gor87}, $\alpha$ is a trivial arc in $X'$. 
This implies that $A$ is a Type 3-1 annulus. 
If $D \subset V$, then $V$ is obtained from $X'$ 
by adding a regular neighborhood of 
a properly embedded simple arc 
in $E(X \setminus \Int \thinspace X')$. 
This implies $A$ is a Type 3-2 annulus.

Suppose that $\partial D$ is separating on $P$. 
Since $A$ is incompressible in $E(V)$, 
$\partial D$ is parallel to neither $a_1$ nor $a_2$ on 
$P$. Hence $\partial D$ is also separating on $\partial V$.  
By compressing $P$ along $D$ we obtain one annulus $P'_1$ and 
one torus $P'_2$. 
In a similar argument as Claim \ref{claim:annulus in a solid torus}, 
we see that $A' \cup P'_1$ bounds a solid torus 
$X'_1$ in $X$. 
See Figure \ref{separating_and_compressing}.  
\begin{figure}[!hbt]
\centering
\includegraphics[width=9cm,clip]{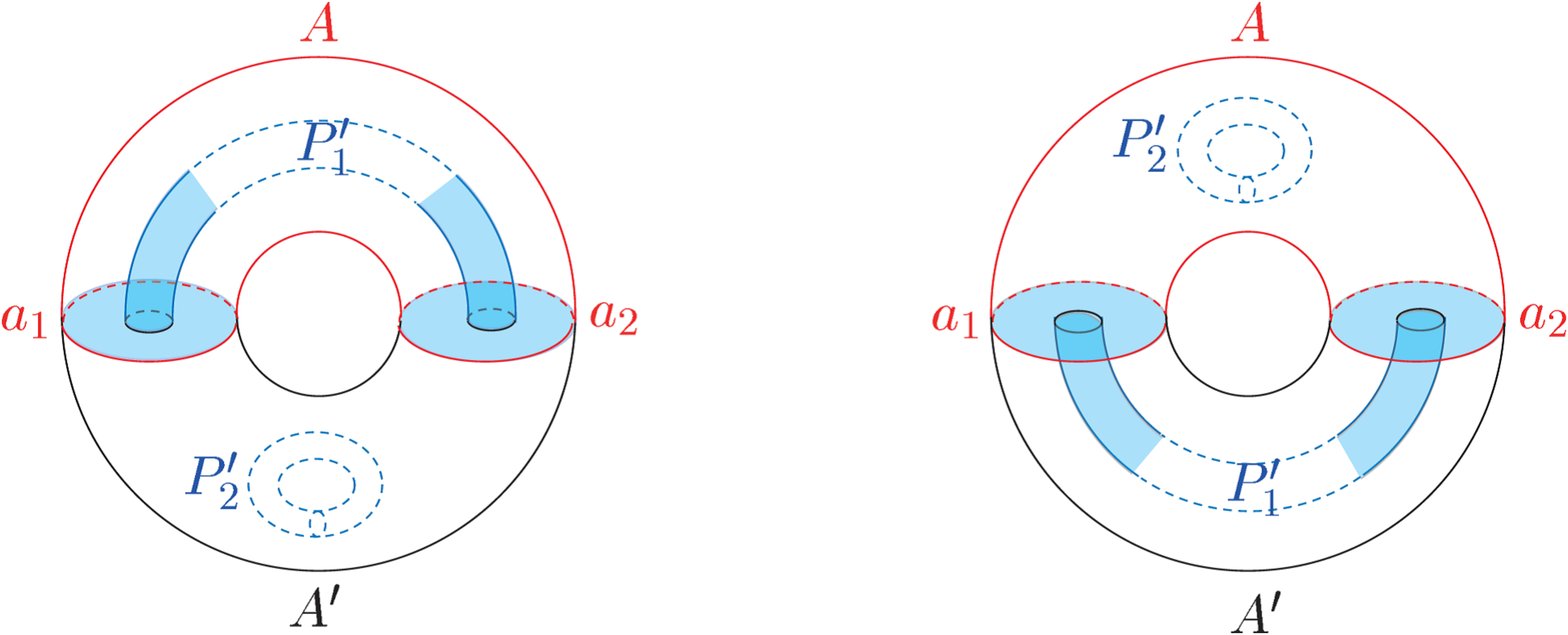}
\caption{}
\label{separating_and_compressing}
\end{figure}
Let $D \subset V$.
Then $P'_2$ bounds 
a solid torus component $X'_2$ 
of $V \setminus \Int \thinspace \Nbd(D; V)$. 
There exists a properly embedded simple arc $\alpha$ 
in $X \setminus \Int \thinspace (X'_1 \cup X'_2)$ connecting $\partial X'_1$ and 
$\partial X'_2$ such that 
$\Nbd (\alpha, X \setminus \Int \thinspace (X'_1 \cup X'_2)) = \Nbd(D; V)$. 
Now using the solid torus $X'_1$ and $\alpha \cup X'_2$ 
it is easy to see that $A$ is a Type 3-2 annulus. 
Let $D \subset E(V)$. 
Then $P'_2$ is contained in $X'_1$. 
When $P'_2$ bounds a solid torus in $X'_1$, then 
we can prove in the same way as above that 
$A$ is a Type 3-1 annulus. 
Suppose $P'_2$ does not bound a solid torus in $X'_1$. 
Then $P'_2$ bounds in $X'_1$ a region $Y$ homeomorphic to 
the exterior of a non-trivial knot in $S^3$. 
Since $\partial D$ is separating on $\partial V$, 
$D \subset E(V)$ is a Type 1 disk. 
If $D$ is a Type 1 disk, both $a_1$ and $a_2$ bound disks in $V$, this contradicts 
the assumption in the beginning of the proof. 
If $D$ is a Type 3 disk, 
one of the components of $E(V) \setminus \Int \thinspace \Nbd (D)
= E(X'_1) \sqcup Y$ must be a solid torus by definition. 
Since $Y$ is not a solid torus, $E(X'_1)$ is a solid torus, i.e. 
$X'_1$ is a standard solid torus in $S^3$. 
Then $A$ is not essential in $E(X'_1)$, so is not in $E(V)$. 
This is a contradiction. 

\vspace{1em} 

Next, suppose that $P$ is contained in $E(X)$. 
Let $K$ be the core of $X$. 
When $K$ is the trivial knot, the above arguments 
immediately implies that $A$ is a Type 3-1 or 3-2 annulus. 
Assume that $K$ is not the trivial knot. 
Since $A$ is essential, $A$ and $A'$ is not parallel in $X$. 
Hence the boundary-slope of $P$ on $\partial X$ is non-integral. 

If $P$ is compressible in $E(X)$, there exists a unique 
annulus component $P'$ of 
the surface obtained by compressing $P$ as in the above argument. 
Since $E(X)$ is boundary-irreducible, $a_1$ and $a_2$ do not bound 
disks in $E(X)$. 
Therefore $P'$ is incompressible. 
Since $A$ determines a cabling annulus of the core of $A'$, 
$P'$ is parallel to $A$ or $A'$. 
The former case is impossible since, if so, 
$A$ is boundary-compressible in $E(V)$.   
In the latter case, applying a similar argument for the 
case of $P \subset X$, we can prove that  
$A$ is a Type 3-1 or 3-2 annulus. 
 
Suppose that $P$ is incompressible in $E(X)$. 
By Lemma \ref{lem:intersection of an essential disk and an annulus}, 
$A' \cup P = \partial V$ is incompressible in $E(V)$, i.e. 
$E(V)$ is boundary-irreducible. 
Similarly, $A \cup P$ is incompressible in $X \cup V$. 
Therefore $A \cup P$ is compressible in 
$V'= E(X \cup V)$. 
It follows that the interior boundary $\partial_- W$ of 
the characteristic compression body $W$ of $V'$ 
is two tori, a single torus or the empty set. 

Assume first that $E(V)$ does not admit an essential torus. 
Then it is clear from the definition that 
$\partial_- W = \emptyset$, i.e. 
$V'$ is also a genus two handlebody. 
Then by Corollary 
\ref{cor:essential twice-punctuerd tori and Eudve-Munoz knots}  
$K$ is an Eudave-Mu\~noz knot. 
This implies that $A$ is a Type 4-1 annulus. 

Finally, assume that $E(V)$ contains essential tori. 
In this case, the interior boundary $\partial_- W$ is 
a single torus or two tori. 
Then, we can re-embed $X \cup V \cup W$ in $S^3$ 
so that $E(X \cup V)$ is a handlebody. 
This implies that $A$ is a Type 4-2 annulus. 
This completes the proof.  
\end{proof}

\noindent {\it Proof of Theorem $\ref{thm:classification of essential annuli}$. } 
Let $A$ be an essential annulus in $E(V)$. 
Then by Lemmas \ref{lem:Classification of Case 1}, 
\ref{lem:Classification of Case 2}, 
\ref{lem:Classification of Case 3} and 
\ref{lem:Classification of Case 4}, 
$A$ belongs to at least one of 
the four Types 1, 2, 3 and 4. 
Moreover, by Lemma \ref{lem:characterization of Type 1 and 2} and 
the definition of Types of annuli in $E(V)$, 
we have the following characterization: 
\begin{enumerate}
\item
If both components of $\partial A$ bound disks in $V$, 
then $A$ is a Type 1 annulus and vice versa. 
\item
If exactly one of the components of $\partial A$ bounds a disk in $V$, 
then $A$ is a Type 2 annulus and vice versa. 
\item
If no component of $\partial A$ bounds a disk in $V$, and 
there exists a compression disk of $\partial V$ in $S^3$ disjoint from $A$, 
then $A$ is a Type 3 annulus and vice versa. 
\item
If no component of $\partial A$ bounds a disk in $V$, and 
there exists no compression disks of $\partial V$ in $S^3$ disjoint from $A$, 
then $A$ is a Type 4 annulus and vice versa. 
\end{enumerate}
The proof is then straightforward. 
\qed

\vspace{1em}

We recall that a handlebody knot $(S^3, V)$ is said to be irreducible  
if it is not 1-decomposable. 
It is equivalent to say that 
$E(V)$ is boundary-irreducible. 
For irreducible genus two handlebody-knots, 
we have a complete classification of 
the essential annuli in their exteriors as follows: 

\begin{corollary}
\label{cor:complete classification of essential annuli}
Let $(S^3, V)$ be an irreducible genus two handlebody-knot. 
Let $A$ be an annulus properly embedded in $E(V)$. 
Then $A$ is essential in $E(V)$ 
if and only if $A$ is a Type $1$, $2$, $3$-$2$, $3$-$3$ or $4$ annulus. 
\end{corollary}
\begin{proof}
The ``only if" part follows from Theorem \ref{thm:classification of essential annuli} and the 
definition of Type 3-1 annulus. 
In fact, if $E(V)$ contains an annulus of type 3-1, the dual disk of the drilling arc $\alpha$ in the 
definition of Type 3-1 annulus gives an essential disk in $E(V)$. 

Let $A$ be a Type $1$, $2$, $3$-$2$, $3$-$3$ or $4$ annulus. 
By definition, each component of $A$ is essential on $\partial V$ and 
$A$ is not parallel to the boundary of $E(V)$. 
Hence by Lemma \ref{lem:annulus with essential boundary is esential}, 
$A$ is essential. 
\end{proof}

\section{Classification of the essential M\"obius bands in genus two handlebody-knot exteriors}
\label{sec:Classification of essential Mobius bands in 
the exteriors of genus two handlebodies embedded in the 3-sphere}

The classification of the essential annuli in the 
exteriors of genus two handlebody-knots provided in the 
previous section directly provides 
a classification of the essential M\"obius bands 
in them. 

Let $Y$ be a solid torus embedded in $S^3$. 
Let $K$ be an $(n,2)$-slope on $\partial Y$ with respect to 
the core of $Y$, where $n$ is an odd integer. 
Set $X = \Nbd(K; S^3)$ and $A = \partial Y \setminus \Int \thinspace X$. 
A Type 3-1 (3-2, respectively) essential annulus in 
the exterior of a genus two handlebody-knot is called 
a {\it Type $3$-$1^*$ annulus} (a {\it Type $3$-$2^*$ annulus}, respectively) 
if it is constructed using 
the above $X$ and $A$ in their definitions. 
An annulus $F$ in the exterior of a genus two handlebody-knot 
is said to be a {\it Type $3$* annulus} if it is 
a Type $3$-$1^*$ or $3$-$2^*$ annulus.

Let $(S^3, V)$ be a genus two handlebody-knot. 
An essential M\"obius band $F$ in $E(V)$ is called 
a {\it Type $1$-$1$, $1$-$2$ and $2$ M\"obius band}, 
respectively, if 
the frontier of its regular neighborhood  
is a Type $3$-$1^*$, $3$-$2^*$ and $4$ annulus, respectively. 
An essential M\"obius band $F$ in the exterior of a genus two handlebody-knot 
is said to be a {\it Type $1$ M\"obius band} if it is 
a Type $1$-$1$ or $1$-$2$ M\"obius band.

\begin{example}
The left-hand side of Figure \ref{example_mobius_band_1} shows 
a Type 1-1 essential M\"obius band $F_1$ 
in the exterior of a genus two 
handlebody-knot $V_1$. 
This example is provided in 
\cite{Mot90} to prove that handlebody-knots are 
not determined by their complements. 
The right-hand side of the same figure shows 
a Type 1-1 essential M\"obius band $F_2$ 
in the exterior of a genus two 
handlebody-knot $V_2$. 
\begin{figure}[!hbt]
\centering
\includegraphics[width=10cm,clip]{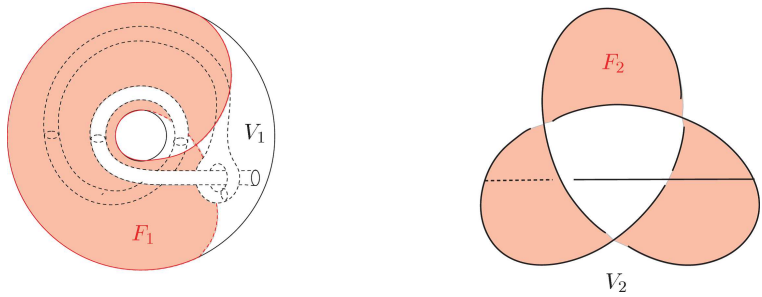}
\caption{Type 1 essential M\"obius bands.}
\label{example_mobius_band_1}
\end{figure}

\end{example}

\begin{theorem}
\label{thm:classification of essential Mobius bands}
Let $(S^3, V)$ be a genus two handlebody-knot. 
Then 
each essential M\"obius band $F$ in the exterior of $V$ 
belongs to exactly one of the above two Types. 
Conversely, each of Type $3^*$ and $4$ essential annuli
in $E(V)$ is the frontier of a regular neighborhood of 
an essential M\"obius band of $E(V)$. 
\end{theorem}
\begin{proof}
Let $A$ be the frontier of a 
regular neighborhood an essential M\"obius band $F$ in $E(V)$. 
Then $A$ satisfies the following:
\begin{enumerate}
\item
$\partial A$ cuts off an annulus $A'$ from $\partial V$; 
\item
$A \cup A'$ bounds a solid torus $Y$ in $E(V)$; 
\item
$\partial A$ is a $(n,2)$-slope with respect to the core of $Y$, where 
$n$ is an odd integer. 
\end{enumerate}
By (1) and Lemmas \ref{lem:Classification of Case 3} and 
\ref{lem:Classification of Case 4}, 
$A$ is a Type 1, 3-1, 3-2 or 4 annulus. 
By (2), Type 1 is impossible. 
Let $A$ be a Type 3-1 or Type 3-2 annulus. 
Then by definition 
there exists a compressing disk $D \subset E(Y)$ for 
$P = \partial V \setminus \Int \thinspace A'$ such that 
$\partial D$ is non-separating on $\partial V$. 
Let $P'$ be the surface 
obtained by compressing $P$ along $D$. 
By (3), the boundary-slope of the 
annulus $P'$ is a $(n,2)$-slope with respect to the core of $Y$. 
This implies that $A$ is a Type $3^*$ annulus. 
Since no essential annulus in $E(V)$ can be both Type 3 and 4 
by Theorem \ref{thm:classification of essential annuli}, 
no essential M\"obius band  in $E(V)$ can be both Type 1 and 2.  

For the other direction, let $A \subset E(V)$ be 
an essential annulus of Type $3'$ or Type $4$. 
Then $\partial A$ cuts off an annulus $A'$ from $\partial V$ 
and $A \cup A'$ bounds a solid torus $Y$ in $E(V)$. 
If $A$ is a Type $3^*$ annulus, by definition, 
$\partial A \subset \partial Y$ consists of $(n,2)$-slopes 
with respect to the core of $X$, where $n$ is an odd integer. 
If $A$ is a Type $4$ annulus, 
by Theorem 
\ref{thm:classification of hyperbolic knots admitting a non-integral toroidal surgery} 
and the definition of Type 4, 
$\partial A \subset \partial X$ also consists of $(n,2)$-slopes 
with respect to the core of $X$, where $n$ is an odd integer. 
Hence in both cases, there exists a M\"{o}bius band 
$F$ properly embedded in $X$ such that 
$\partial F$ is the core of the annulus $A'$. 
Since the frontier of $F$ is 
isotopic to $A$ in $E(V)$, $F$ is essential in $E(V)$. 
This completes the proof. 
\end{proof}

As a direct corollary of 
Theorem \ref{thm:classification of essential Mobius bands}, 
we have the following: 
\begin{corollary}
\label{cor:classification of essential Mobius bands}
Let $(S^3, V)$ be a genus two handlebody-knot. 
Then there exists a one-to-one correspondence between 
the set of isotopy classes of essential M\"obius bands 
in $E(V)$ and 
the set of isotopy classes of 
Type $3^*$ or $4$ essential annuli in $E(V)$. 
\end{corollary}

As for essential annuli, 
we have a complete classification of 
the essential M\"obius bands in the exteriors 
of irreducible genus two handlebody-knots. 
\begin{corollary}
\label{cor:complete classification of essential Mobius bands}
Let $(S^3, V)$ be an irreducible genus two handlebody-knot. 
Let $F$ be a M\"obius band properly embedded in $E(V)$. 
Then $F$ is essential in $E(V)$ 
if and only if $F$ is a Type $1$-$2$ or $2$ M\"obius band. 
\end{corollary}
\begin{proof}
This follows immediately from 
Corollary \ref{cor:complete classification of essential annuli} and 
Theorem \ref{thm:classification of essential Mobius bands}. 
\end{proof}

\section{Classification of the essential tori in genus two handlebody-knot exteriors}
\label{sec:Classification of essential tori in 
the exteriors of genus two handlebodies embedded in the 3-sphere}

Let $(S^3, V)$ be a handlebody-knot 
and let $T$ be a torus properly embedded in $E(V)$. 
A peripherally compressing annulus $A$ for $T$ in $E(V)$ 
is called, in particular, a {\it meridionally compressing annulus} 
if $A \cap \partial V$ bounds an essential disk in $V$. 
We say that $T$ is {\it meridionally compressible} 
if it admits a meridionally compressing annulus. 
Otherwise, $T$ is said to be {\it meridionally incompressible}. 

\begin{theorem}
\label{thm:classification of essential tori}
Let $(S^3, V)$ be a genus two handlebody-knot. 
Let $T$ be an essential torus in $E(V)$. 
Then the following holds: 
\begin{enumerate}
\item
If $T$ is meridionally compressible, 
then there exists a Type $1$ essential annulus 
$A$ in $E(V)$ such that 
$\partial A$ cuts off from $\partial V$ an annulus $A'$ 
so that $A \cup A'$ is isotopic to $T$. 
\item
If $T$ is not meridionally compressible but peripherally compressible, 
then there exists a Type $3$-$1$ or $3$-$2$ essential annulus 
$A$ in $E(V)$ such that 
$\partial A$ cuts off from $\partial V$ an annulus $A'$ 
so that $A \cup A'$ is isotopic to $T$. 
\item
If $T$ is peripherally incompressible, 
then there exists a handlebody-knot $(S^3, V')$ and 
a solid torus $X$ in $E(V')$ 
such that 
$E(V' \cup X)$ does not contain an essential annulus $A$ 
with $A \cap \partial V' \neq \emptyset$ and $A \cap \partial X \neq \emptyset$, 
and that there exists a re-embedding $h: E(X) \to S^3$ so that 
$h(V') = V$ and $h(\partial E(V)) = T$. 	
\end{enumerate}
\end{theorem}
\begin{proof}
Let $Y \subset S^3$ be the solid torus bounded by $T$. 
Since $T$ is incompressible in $E(V)$, $Y$ contains $V$. 

Assume that there exists a peripherally compressing annulus 
$\hat{A}$ for $T$. 
Let $A \subset E(V)$ be the annulus obtained by peripherally 
compressing $T$ along $\hat{A}$. 
Since $T$ is essential, it follows 
from a similar argument of 
Claim \ref{claim:there is no separating loop} 
in the proof of Lemma 
\ref{lem:satellite knots admitting 2-pounc. tori with non-integral boundary slope} 
that $A$ is also essential. 
Let $A'$ be the annulus component of $\partial V$ cut off by $\partial A$. 
We note that $T$ is ambient isotopic to $A \cup A'$. 
Then it is immediate from Lemma 
\ref{lem:characterization of Type 1 and 2} that 
$\hat{A}$ is a meridionally compressing annulus if and only if 
$A$ is a Type 1 annulus. 
If $\hat{A}$ is not a meridionally compressing annulus, 
by Lemmas \ref{lem:Classification of Case 3} and 
\ref{lem:Classification of Case 4}, 
$A$ is a Type 3-1, 3-2 or 4 annulus. 
However, Type 4 is impossible, since, if so, 
$A \cup A'$ bounds a solid torus in $E(V)$, which implies 
that $T$ is compressible. 

Next, assume that $T$ is peripherally incompressible. 
Then $Y \setminus \Int \thinspace V$ 
does not contain an essential annulus $A$ 
with $A \cap \partial V \neq \emptyset$ and 
$A \cap \partial Y \neq \emptyset$. 
We can re-embed $Y$ into $S^3$ by 
a map $i$ so that $X = E(i(Y))$ is a solid torus. 
The assertion is now easily seen 
by settting $V' = i(V)$. 
\end{proof}

\appendix 
\section*{Appendix A (by Cameron Gordon)}


\begin{proof}[Proof of Lemma $\ref{lem:4-punctured spheres}$]
Let $K$ be a knot in $S^3$, with exterior $E(K)$.
The lemma is clearly true if $K$ is trivial, so assume that $K$ is non-trivial. 
Let $P$ be a properly embedded incompressible 
4-punctured sphere in $E(K)$ with integral boundary slope $\alpha$. 

We will assume familiarity with the terminology of labeled fat vertex intersection graphs, as 
described for example in \cite{GL89}.

Let $E(K)(\alpha) = E(K) \cup V_\alpha$ be the closed manifold obtained by $\alpha$-Dehn filling
on $E(K)$, where $V_\alpha$ is the filling solid torus.
We may cap off the components of $\partial P$ with meridian disks $v_1,v_2,v_3,v_4$ of 
$V_\alpha$ (numbered in order along $V_\alpha$) to get a 2-sphere $\widehat P \subset E(K)(\alpha)$.
By \cite{Gab87}, if we put $K$ in thin position then there is a level 2-sphere $\widehat Q\subset S^3$, 
with corresponding meridional planar surface $Q = \widehat Q \cap E(K) \subset E(K)$, such 
that the intersection graphs $\Gamma_P$ and $\Gamma_Q$ in $\widehat P$ and $\widehat Q$ 
respectively, defined in the usual way by the arc components of $P\cap Q$, have no monogon faces.
Note that $v_1,v_2,v_3,v_4$ are the (fat) vertices of $\Gamma_P$.

Since $H_1 (S^3) = 0$, $\Gamma_P$ does not represent 
all types \cite{Par90}, and hence by \cite{GL89} 
$\Gamma_Q$ contains a Scharlemann cycle $\sigma$. 
Let $f$ be the disk face of $\Gamma_Q$ bounded by $\sigma$ and let $k\ge 2$ be the number
of edges in $\sigma$.
Since $P$ is incompressible we can assume by standard arguments that 
$(\Int \thinspace f)\cap P = \emptyset$.
Without loss of generality $\sigma$ is a (12)-Scharlemann cycle.
The edges of $\sigma$ give rise to $k$ corresponding ``dual'' edges of $\Gamma_P$, joining  
vertices~$v_1$ and $v_2$. 
They thus divide $\widehat P$ into $k$ segments.
\renewcommand{\qed}{}
\end{proof}

\begin{claim}\label{clm1}
Vertices $v_3$ and $v_4$ of $\Gamma_P$ lie in the same segment.
\end{claim}

\begin{proof}
Let $q$ be the number of components $\partial Q$, which is equal to the valency of the vertices of 
$\Gamma_P$.
Suppose $v_3$ and $v_4$ lie in different segments.
Then of the $q$ edges of $\Gamma_P$ incident to  $v_3$,  $a_1$ go to  $v_1$ 
and $a_2$ to  $v_2$,
where $a_1 + a_2 = q$.
Similarly $b_1$ edges  at  $v_4$ go to  $v_1$ and $b_2$ to  $v_2$, where $b_1+ b_2 = q$.
It follows that either $a_1 + b_1$ or $a_2 +b_2$ is $\ge q$, a contradiction.
\end{proof}

By Claim~\ref{clm1} there is a disk $D\subset \widehat P$ containing fat vertices~$v_1$ and $v_2$ and the $k$ 
edges dual to $\sigma$,  and disjoint from $v_3$ and $v_4$.

Let $H_{12}$ be that part of $V_\alpha$ that runs between fat vertices~$v_1$ and $v_2$ of $\Gamma_P$.
Let $\widehat X$  be a regular neighborhood of $D\cup H_{12} \cup f$, pushed slightly off $\widehat P$.
Then $\widehat X$ is a punctured lens space whose fundamental group has order~$k$.
The 2-sphere $\partial\widehat X$ meets $V_\alpha$ in two meridian disks that are nearby 
parallel copies of  $v_1$ and $v_2$. 
Let $A$ be the annulus $\partial\widehat X\cap E(K)$.
Then $A$ separates $E(K)$ into $X$ and $Y$, say, where $X\subset\widehat X$.
Note that $\widehat X$ is obtained by attaching the 2-handle $H_{12}$ to $X$, and that $P\subset Y$.

\begin{claim}\label{clm2}
$A$ is essential in $E(K)$.
\end{claim}

\begin{proof}
Clearly $A$ is incompressible in $E(K)$.

Suppose $A$ is boundary parallel in $E(K)$.
Then either $X$ or $Y$ is homeomorphic to $A\times I$ (with $A$ corresponding to 
$A\times\{0\}$).
In the first case, $\widehat X$ is homeomorphic to $B^3$, a contradiction. 
In the second case, 
%
%
$P$ compresses in $Y$, and therefore in $E(K)$, again a contradiction.
\end{proof}

Claim~\ref{clm2} implies that $K$ is a cable knot with cabling annulus $A$.
Since $P$ has the same boundary slope as $A$ it is easy to show that this is impossible.\qed

\appendix 
\section*{Appendix B: Decomposition of handlebody-knots by 2-decomposing spheres}
\renewcommand{\thesection}{B}


In this appendix, we provide a unique decomposition theorem of handlebody-knots of 
of arbitrary genus by decomposing spheres, which is a generalization of \cite{IKO12}. 
This is achieved by focusing only on a generalization of Type 1 annuli defined in Section 
\ref{sec:Classification of essential annuli in the exteriors of genus two handlebodies embedded in the 3-sphere} 
for higher genus case.

A $2$-decomposing sphere $S$ in $S^3$ is 
called a {\it knotted handle decomposing sphere} for a
handlebody-knot $(S^3, V)$ if 
$S \cap V$ consists of two parallel essential disks in $V$, and 
$S \cap E(V)$ is an essential annulus in $E(V)$. 

Let $(S^3, V)$ be a handlebody-knot and $S$ be its knotted 
handle decomposing sphere. 
Then $S \cap \partial V$ cuts off an annulus $A$ from 
$\partial V$.  
Let $T$ be an essential torus in $E(V)$ obtained by tubing 
$S \cap E(V)$ along $A$. 
Let $\hat{A}$ be a meridionally compression annulus for $T$. 
Then by annulus-compressing $T$ along $\hat{A}$, we get a new knotted 
handle decomposing sphere $S'$. 
We say that 
{\it $S'$ is obtained from $S$ by an annulus-move along $A$}. 

A set ${S_1, \ldots , S_n}$ of knotted handle decomposing 
spheres for a handlebody-knot $(S^3, V)$ is said to be {\it unnested} 
if each sphere $S_i$ bounds a 3-ball $B_i$ in $S^3$ so that 
$B_i \cap V \cong B^3$ ($1 \leqslant i \leqslant n$) and 
$B_i \cap B_j = \emptyset$ ($1 \leqslant i < j \leqslant n$). 
We remark that 
a maximal unnested set of knottted handle decomposing spheres 
always exists by the Kneser-Haken finiteness theorem 
\cite{Kne29, Hak68}. 
Moreover, the following is proved in \cite{IKO12}. 

\begin{theorem}[\cite{IKO12}]
\label{thm:IKO12}
Let $(S^3, V)$ be a handlebody-knot such that 
$E(V)$ is boundary-irreducible. 
Then $(S^3, V)$ admits a unique maximal unnested 
set of knotted handle decomposing spheres up to
isotopies and annulus-moves.
\end{theorem}

In the following, we see that we can 
remove from the above theorem 
the assumption that $E(V)$ is boundary-irreducible. 
\begin{theorem}
\label{thm:uniqueness of 2-decompositions}
Every handlebody-knot $(S^3, V)$ admits a unique maximal unnested 
set of knotted handle decomposing spheres up to
isotopies and annulus-moves.
\end{theorem}

\begin{lemma}
\label{lem:maximal system and another torus}
Let $\{T_1, T_2, \ldots, T_n\}$ be a maximal set of mutually disjoint, 
mutually non-parallel, essential, meridional-compressible tori 
in $E(V)$ satisfying the following: 
\begin{itemize}
\item
For each $i=1,2,\ldots, n$, 
let $Y_i$ be the region in $S^3$ spanned by 
$T_i$ such that $Y_i \cap V = \emptyset$. 
Then $Y_i \cap Y_j = \emptyset$ for 
$1 \leqslant i < j \leqslant n$; and 
\item
The core $K_i$ of $E(Y_i)$ is a prime knot. 
\end{itemize}
Then any essential, meridional-compressible torus $T$ in $E(V)$ 
can be isotoped so that $T \cap T_i = \emptyset$ for all $i$. 
\end{lemma}
\begin{proof}
Assume for contradiction that 
$T \cap ( \bigcup_{i=1}^n T_i ) \neq \emptyset$ after 
minimizing the number of components of $T \cap ( \bigcup_{i=1}^n T_i )$ 
by an isotopy. 
Let $T_i$ intersects $T$. 
Let $A$ be a component of $T \cap Y_i$. 
Then by Lemma \ref{lem:BZ85} 
and the assumption that $K_i$ is a prime knot, 
$A$ is a cabling annulus for $K_i$. 
Hence $A$ intersects a meridionally compressing annulus $A_i$ 
for $T_i$. 
It follows that $A_i \cap T$ consists of non-empty proper arcs 
with end points on $\partial A_i \setminus \partial V$. 

Let $\delta \subset A_i$ be the disk cut off from $A_i$ by an outermost
arc $\alpha$ of $A_i \cap T$ in $A_i$. 
Let $A'$ be the component of $T \cap E(Y_i)$ containing $\alpha$. 
See Figure \ref{fig:proof_of_unique_2-decomp}. 
\begin{figure}[!hbt]
\centering
\includegraphics[width=10cm,clip]{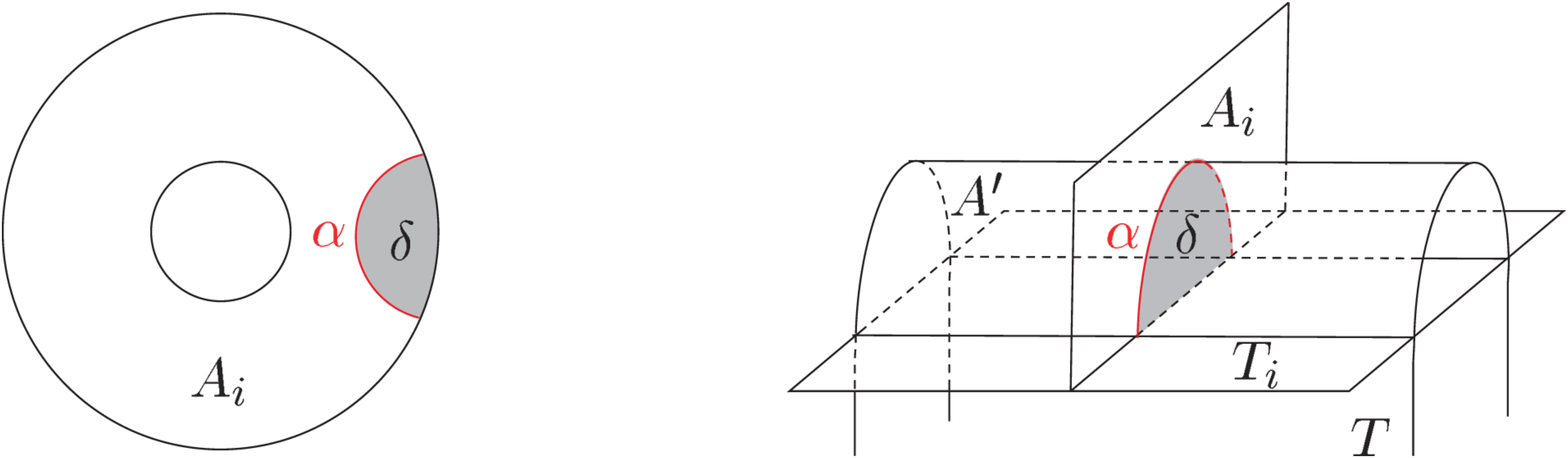}
\caption{}
\label{fig:proof_of_unique_2-decomp}
\end{figure}
By boundary-compressing $A'$ along $\delta$, we get a disk $D$ whose 
boundary bounds a disk $D'$ on $T_i$. 
Since a solid torus is irreducible, $D \cap D'$ bounds a 3-ball in 
$E(Y_i)$. 
This implies that $A'$ can 
be isotoped to $E(Y_i)$ in $E(V) \cap E(Y_i)$. 
This contradicts the minimality of $\#(T \cap ( \bigcup_{i=1}^n T_i ))$. 
\end{proof}

\begin{lemma}
\label{lem:maximal system spans disjoint knot complements}
Let $\{T_1, T_2, \ldots, T_n\}$ be a maximal set of mutually disjoint, 
mutually non-parallel, essential tori in $E(V)$ such that 
there exist peripherally compressing annuli $A_i$ for $T_i$ 
$(1 \leqslant i \leqslant n)$ 
with $A_j \cap A_k = \emptyset$, $A_j \cap T_k = \emptyset$ 
for $1 \leqslant j < k \leqslant n$. 
For each $i=1,2,\ldots, n$, 
let $Y_i$ be the region in $S^3$ spanned by 
$T_i$ such that $Y_i \cap V = \emptyset$. 
Then $Y_i \cap Y_j = \emptyset$ for 
$1 \leqslant i < j \leqslant n$. 
\end{lemma}
\begin{proof}
If $Y_i \cap Y_j \neq \emptyset$ for some $i, j$, 
we may assume without loss of generality that 
$Y_i \subset Y_j$ since every torus embedded in $S^3$ is 
separating. 
However, this is impossible since it is assumed that 
the compressing annulus $A_j$ does not intersect $T_i$. 
\end{proof}

\noindent {\it Proof of Theorem $\ref{thm:uniqueness of 2-decompositions}$. }
Let $\mathcal{S} = \{ S_1, S_2, \ldots, S_n \}$  and 
$\mathcal{S}' = \{ S'_1, S'_2, \ldots, S'_n \}$ be 
maximal unnested sets of 
knotted handle 2-decomposing spheres for 
a handlebody-knot $(S^3, V)$. 
Since they are unnested, 
each sphere $S_i$ ($S'_i$, respectively) bounds a 3-ball 
$B_i$ ($B'_i$) in $S^3$ such that 
$B_i \cap V \cong B^3$ 
($B'_i \cap V \cong B^3$, respectively) 
($1 \leqslant i \leqslant n$) and 
$B_i \cap B_j = \emptyset$ 
($B'_i \cap B'_j = \emptyset$, respectively) ($1 \leqslant i < j \leqslant n$). 
Each sphere $S_i$ 
($S'_i$, respectively) separates an annulus 
$A_i$ ($A'_i$, respectively) from $\partial V$.  
Let $T_i$ ($T'_i$, respectively) be an essential torus 
in $E(V)$ obtained by tubing 
$S_i \cap E(V)$ ($S'_i \cap E(V)$) along $A_i$ ($A'_i$, respectively). 
Let $Y_i$ ($Y'_i$, respectively) be the region in $S^3$ spanned by 
$T_i$ ($T'_i$, respectively) such that $Y_i \cap V = \emptyset$ 
($Y'_i \cap V = \emptyset$, respectively). 
It is easy to check that the set 
$\mathcal{T} = \{ T_1, T_2, \ldots, T_n \}$ 
(resp. $\mathcal{T}' = \{ T'_1, T'_2, \ldots, T'_n \}$, respectively) 
satisfies the assumption of Lemma 
\ref{lem:maximal system spans disjoint knot complements}. 
Therefore we have 
$Y_i \cap Y_j = \emptyset$ 
($Y'_i \cap Y'_j = \emptyset$, respectively) for 
$1 \leqslant i < j \leqslant n$. 
Moreover, by Schubert's theorem \cite{Sch49}, 
the core $K_i$ ($K'_i$, respectively) of $E(Y_i)$ 
($E(Y'_i)$, respectively) is prime for $1 \leqslant i \leqslant i$. 
Hence by Lemma 
\ref{lem:maximal system and another torus} that  
we have $\mathcal{T} = \mathcal{T}'$. 
This implies that $\mathcal{S}'$ is obtained by 
at most $n$ annulus-moves from $\mathcal{S}$.

\end{document}